\documentclass[12pt,reqno]{amsart}

\usepackage[cp1251]{inputenc}
\usepackage[T2A]{fontenc}
\usepackage[english]{babel}
\usepackage{amsmath, amsthm, amscd, amsfonts, amssymb, graphicx, color}
\usepackage[bookmarksnumbered, colorlinks, plainpages]{hyperref}
\usepackage{geometry}
\usepackage{cite}
\newtheorem{theorem}{Theorem}[section]

\newtheorem{lemma}[theorem]{Lemma}
\newtheorem{proposition}{Proposition}

\theoremstyle{definition}

\newtheorem{remark}{Remark}

\newtheorem{condition}[theorem]{Condition} 

\author[Valerii Los and Aleksandr Murach]
{Valerii Los and Aleksandr Murach}

\title[Isomorphism theorems for parabolic problems]
{Isomorphism theorems for some parabolic initial-boundary value problems in H\"ormander spaces}

\address{National Technical University of Ukraine Igor Sikorsky Kyiv Polytechnic Institute,
Prospect Peremohy 37, 03056, Kyiv-56, Ukraine}

\email{v\_los@yahoo.com}

\address{Institute of Mathematics, National Academy of Sciences of Ukraine, 3 Tereshchenkivs'ka, Kyiv, 01004, Ukraine
and Chernihiv National Pedagogical University, Het'mana Polubotka str. 53, 14013 Chernihiv, Ukraine}

\email{murach@imath.kiev.ua}



\begin{document}

\maketitle

\begin{abstract}
 In H\"ormander inner product spaces, we investigate initial-boundary value problems for an arbitrary second order parabolic partial differential equation and the Dirichlet or a general first-order boundary conditions. We prove that the operators corresponding to these problems are isomorphisms between appropriate H\"ormander spaces. The regularity of the functions which form these spaces is characterized by a pair of number parameters and a function parameter varying regularly at infinity in the sense of Karamata. Owing to this function parameter, the H\"ormander spaces describe the regularity of functions more finely than the anisotropic Sobolev spaces.
 \end{abstract}
\vspace{1cm}

\section{Introduction}\label{8sec1}

The modern theory of general parabolic initial-boundary problems has been
developed for the classical scales of H\"older--Zygmund and Sobolev function
spaces \cite{AgranovichVishik64, Friedman64, LadyzhenskajaSolonnikovUraltzeva67, LionsMagenes72ii, Zhitarashu85, Ivasyshen90, Eidelman94, Lunardi1995, ZhitarashuEidelman98}. The central result of this theory are the theorems on
well-posedness by Hadamard of these problems on appropriate pairs of these
spaces. For applications, especially to the spectral theory
of differential operators, inner product Sobolev spaces play a special role.

In 1963 H\"ormander \cite{Hermander63} proposed a broad and meaningful generalization
of the Sobolev spaces in the framework of Hilbert spaces. He introduced the spaces
$$
\mathcal{B}_{2,\mu}:=\bigl\{ w\in\mathcal{S}'(\mathbb{R}^{k})
:\mu(\xi)\widehat{w}(\xi)\in L_2(\mathbb{R}^{k},\,d\xi)\bigr\},
$$
for which a general Borel measurable weight function $\mu:\mathbb{R}^{k}\rightarrow(0,\infty)$ serves as an index of regularity of a distribution $w$. (Here, $\widehat{w}$ denotes the Fourier transform
of $w$.) These spaces and their versions within the category of normed spaces (so called
spaces of generalized smoothness) have found various applications to analysis and partial differential equations \cite{VolevichPaneah65, Lizorkin86, Paneah00, Jacob010205, Triebel01, FarkasLeopold06, NicolaRodino10, MikhailetsMurach14, MikhailetsMurach15}.

Recently Mikhailets and Murach \cite{MikhailetsMurach06UMJ2, MikhailetsMurach06UMJ3, MikhailetsMurach07UMJ5, Murach07UMJ6, MikhailetsMurach08UMJ4} have built a theory of solvability of general elliptic systems and elliptic boundary-value problems on Hilbert scales of spaces $H^{s;\varphi}:=\mathcal{B}_{2,\mu}$ for which the index of regularity is of the form
$$
\mu(\xi):=(1+|\xi|^{2})^{s/2}\varphi((1+|\xi|^{2})^{1/2}).
$$
Here, $s$ is a real number, and $\varphi$ is a function varying slowly at infinity
in the sense of Karamata \cite{Karamata30a}. This theory is based on the method
of interpolation with a function parameter between Hilbert spaces, specifically
between Sobolev spaces. This allows Mikhailets and Murach to deduce theorems about solvability of elliptic systems and elliptic problems from the known results on the solvability of elliptic equations in Sobolev spaces. This theory is set force in \cite{MikhailetsMurach14, MikhailetsMurach12BJMA2}.

Generally, the method of interpolation between normed spaces proved to be very useful in the theory of elliptic \cite{Berezansky65, LionsMagenes72i, Triebel95} and parabolic \cite{LionsMagenes72ii, Lunardi1995} partial differential equations. Specifically, Lions and Magenes \cite{LionsMagenes72ii} systematically used the interpolation with a number (power) parameter between Hilbert spaces in their theory of solvability of parabolic initial-boundary value problems on a complete scale of anisotropic Sobolev spaces. Using the more flexible method of interpolation with a function parameter between Hilbert spaces, Los, Mikhailets, and Murach \cite{LosMurach13MFAT2, LosMikhailetsMurach17CPAA} proved theorems on solvability of semi-homogeneous parabolic problems in $2b$-anisotropic H\"ormander spaces $H^{s,s/(2b);\varphi}$, where $2b$ is a parabolic weight and where the parameters $s$ and $\varphi$ are the same as those in the above mentioned elliptic theory. These problems were considered in the case of homogeneous initial conditions (Cauchy data).

The purpose of this paper is to establish the well-posedness of inhomogeneous parabolic problems on appropriate pairs of the H\"ormander spaces, i.e. to prove new isomorphism theorems for these problems. We consider the problems that consist of a general second order parabolic partial differential equation, the Dirichlet boundary condition or a general first order boundary condition, and the Cauchy datum. We deduce these isomorphism theorems from Lions and Magenes' result \cite{LionsMagenes72ii} with the help of the interpolation with a function parameter between anisotropic Sobolev spaces. The use of this method in the case of inhomogeneous parabolic problems meets additional difficulties connected with the necessity to take into account quite complex compatibility conditions imposed on the right-hand sides of the problem. The model case of initial boundary-value problems for heat equation is investigated in \cite{Los15UMG5}.

\section{Statement of the problem}\label{8sec2}

We arbitrarily choose an integer $n\geq2$ and a real number $\tau>0$. Let $G$ be a bounded domain in $\mathbb{R}^{n}$ with an infinitely smooth boundary $\Gamma:=\partial G$. We put
$\Omega:=G\times(0,\tau)$ and $S:=\Gamma\times(0,\tau)$; so, $\Omega$ is an open cylinder in $\mathbb{R}^{n+1}$, and $S$ is its lateral boundary. Then $\overline{\Omega}:=\overline{G}\times[0,\tau]$ and
$\overline{S}:=\Gamma\times[0,\tau]$ are the closures of $\Omega$ and $S$ respectively.

In $\Omega$, we consider a parabolic second order partial differential equation
\begin{equation}\label{8f1}
\begin{gathered}
Au(x,t)\equiv\partial_{t}u(x,t)+
\sum_{|\alpha|\leq2}a_{\alpha}(x,t)\,D^\alpha_x
u(x,t)=f(x,t)\\
\mbox{for all}\;\;x\in G\;\;\mbox{and}\;\;t\in(0,\tau).
\end{gathered}
\end{equation}
Here and below, we use the following notation for partial derivatives: $\partial_t:=\partial/\partial t$ and $D^\alpha_x:=D^{\alpha_1}_{1}\dots D^{\alpha_n}_{n}$, where $D_{j}:=i\,\partial/\partial{x_j}$,  $x=(x_1,\ldots,x_n)\in\mathbb{R}^{n}$, and $\alpha:=(\alpha_1,\ldots,\alpha_n)$ with $0\leq\alpha_1,...,\alpha_n\in\mathbb{Z}$ and $|\alpha|:=\alpha_1+\cdots+\alpha_n$. We suppose that all the coefficients $a_{\alpha}$ of $A$ belong to the space $C^{\infty}(\overline{\Omega})$. In the paper, all functions and distributions are supposed to be complex-valued, so we consider complex function spaces.

We suppose that the partial differential operator $A$ is Petrovskii parabolic on $\overline{\Omega}$, i.e. it satisfies the following condition (see, e.g. \cite[Section~9, Subsection~1]{AgranovichVishik64}):

\begin{condition}\label{8cond1}
For arbitrary $x\in\overline{G}$, $t\in[0,\tau]$, $\xi=(\xi_{1},\ldots,\xi_{n})\in\mathbb{R}^{n}$, and $p\in\mathbb{C}$ with $\mathrm{Re}\,p\geq0$, the inequality
\begin{equation*}
p+\sum_{|\alpha|=2} a_{\alpha}(x,t)\,\xi_{1}^{\alpha_{1}}\cdots\xi_{n}^{\alpha_{n}}
\neq0\quad\mbox{holds whenever}\quad|\xi|+|p|\neq0.
\end{equation*}
\end{condition}

In the paper, we investigate the initial-boundary value problem that consists of the parabolic equation~\eqref{8f1}, the initial condition
\begin{equation}\label{8f3}
u(x,0)=h(x)\quad\mbox{for all}\;\;x\in G,
\end{equation}
and the zero-order (Dirichlet) boundary condition
\begin{equation}\label{8f2}
u(x,t)=g(x,t)
\quad\mbox{for all}\;\;x\in\Gamma\;\;\mbox{and}\;\;t\in(0,\tau)
\end{equation}
or the first order boundary condition
\begin{equation}\label{8f2n}
\begin{gathered}
Bu(x,t)\equiv\sum_{j=1}^{n}b_j(x,t)D_ju(x,t)+b_0(x,t)u(x,t)=g(x,t)\\
\mbox{for all}\;\;x\in\Gamma\;\;\mbox{and}\;\;t\in(0,\tau).
\end{gathered}
\end{equation}

As to \eqref{8f2n}, we assume that all the coefficients $b_0$, $b_1$, ..., $b_n$ of $B$ belong to $C^{\infty}(\overline{S})$ and that
$B$ covers $A$ on $\overline{S}$ \cite[Section~9, Subsection~1]{AgranovichVishik64}. The latter assumption means the fulfilment of the following:

\begin{condition}\label{8cond2}
Choose arbitrarily $x\in\Gamma$, $t\in[0,\tau]$, vector
$\eta=(\eta_1,\dots,\eta_n)\in\mathbb{R}^{n}$ tangent to the boundary $\Gamma$ at the point $x$, and number $p\in\mathbb{C}$ with $\mathrm{Re}\,p\geq0$ so that $|\eta|+|p|\neq0$. Let $\nu(x)=(\nu_1(x),\dots,\nu_n(x))$ be the unit vector of the inward normal to $\Gamma$ at $x$. Then:
\begin{itemize}
\item[a)] the inequality
$
\sum_{j=1}^{n}b_j(x,t)\nu_j(x)\neq0
$
holds true;
\item[b)] the number
$$
\zeta=-\biggl(\,\sum\limits_{j=1}^{n}b_j(x,t)\eta_j\biggr)
\biggl(\,\sum\limits_{j=1}^{n}b_j(x,t)\nu_j(x)\biggr)^{-1}
$$
is not a root of the polynomial
$$
p+\sum_{|\alpha|=2} a_{\alpha}(x,t)\,(\eta_{1}+\zeta\nu_{1}(x))^{\alpha_{1}}\cdots
(\eta_{n}+\zeta\nu_{n}(x))^{\alpha_{n}}\quad\mbox{of}\;\;
\zeta\in\mathbb{C}.
$$
\end{itemize}
\end{condition}

It is useful to note that if all the coefficients $b_1$,..., $b_n$ are real-valued, then part~b) of Condition~\ref{8cond2} is satisfied. This follows directly from Condition~\ref{8cond1}.

Thus, we examine both the parabolic problem \eqref{8f1}, \eqref{8f3}, \eqref{8f2} and the parabolic problem \eqref{8f1}, \eqref{8f3}, \eqref{8f2n}. We investigate them in appropriate H\"ormander inner product spaces considered in the next section.

\section{H\"ormander spaces}\label{8sec3}

Among the normed function spaces $\mathcal{B}_{p,\mu}$ introduced by H\"ormander in \cite[Section~2.2]{Hermander63}, we use the inner product spaces $H^{\mu}(\mathbb{R}^{k}):=\mathcal{B}_{2,\mu}$ defined over $\mathbb{R}^{k}$, with $1\leq k\in\mathbb{Z}$. Here,  $\nobreak{\mu:\mathbb{R}^{k}\rightarrow(0,\infty)}$ is an arbitrary Borel measurable function that satisfies the following condition: there exist positive numbers $c$ and $l$ such that
$$
\frac{\mu(\xi)}{\mu(\eta)}\leq
c\,(1+|\xi-\eta|)^{l}\quad\mbox{for all}\quad \xi,\eta\in\mathbb{R}^{k}.
$$

By definition, the (complex) linear space $H^{\mu}(\mathbb{R}^{k})$ consists of all tempered distributions $w\in\mathcal{S}'(\mathbb{R}^{k})$ whose Fourier transform $\widehat{w}$ is a locally Lebesgue integrable function subject to the condition
\begin{equation*}
\int\limits_{\mathbb{R}^{k}}\mu^{2}(\xi)\,|\widehat{w}(\xi)|^{2}\,d\xi
<\infty.
\end{equation*}
The inner product in $H^{\mu}(\mathbb{R}^{k})$ is defined by the formula
\begin{equation*}
(w_1,w_2)_{H^{\mu}(\mathbb{R}^{k})}=
\int\limits_{\mathbb{R}^{k}}\mu^{2}(\xi)\,\widehat{w_1}(\xi)\,
\overline{\widehat{w_2}(\xi)}\,d\xi,
\end{equation*}
where $w_1,w_2\in H^{\mu}(\mathbb{R}^{k})$. This inner product induces the norm
$$
\|w\|_{H^{\mu}(\mathbb{R}^{k})}:=(w,w)^{1/2}_
{H^{\mu}(\mathbb{R}^{k})}.
$$
According to \cite[Section~2.2]{Hermander63}, the space $H^{\mu}(\mathbb{R}^{k})$ is Hilbert and separable with respect to this inner product. Besides that, this space is continuously embedded in the linear topological space $\mathcal{S}'(\mathbb{R}^{k})$ of tempered distributions on $\mathbb{R}^{k}$, and the set $C^{\infty}_{0}(\mathbb{R}^{k})$ of test functions on $\mathbb{R}^{k}$ is dense in $H^{\mu}(\mathbb{R}^{k})$ (see also H\"ormander's monograph \cite[Section~10.1]{Hermander83}). We will say that the function parameter $\mu$ is the regularity index for the space $H^{\mu}(\mathbb{R}^{k})$ and its versions $H^{\mu}(\cdot)$.

A version of $H^{\mu}(\mathbb{R}^{k})$ for an
arbitrary nonempty open set $V\subset\mathbb{R}^{k}$ is introduced in the standard way. Namely,
\begin{gather}\notag
H^{\mu}(V):=\bigl\{w\!\upharpoonright\!V:\,
w\in H^{\mu}(\mathbb{R}^{k})\bigr\},\\
\|u\|_{H^{\mu}(V)}:= \inf\bigl\{\|w\|_{H^{\mu}(\mathbb{R}^{k})}:\,w\in
H^{\mu}(\mathbb{R}^{k}),\;u=w\!\upharpoonright\!V\bigr\}, \label{8f40}
\end{gather}
where $u\in H^{\mu}(V)$. Here, as usual, $w\!\upharpoonright\!V$ stands for the restriction of the distribution $w\in H^{\mu}(\mathbb{R}^{k})$ to the open set~$V$. In other words, $H^{\mu}(V)$ is the factor space of the space $H^{\mu}(\mathbb{R}^{k})$ by its subspace
\begin{equation}\label{8f41}
H^{\mu}_{Q}(\mathbb{R}^{k}):=\bigl\{w\in
H^{\mu}(\mathbb{R}^{k}):\, \mathrm{supp}\,w\subseteq Q\bigr\} \quad\mbox{with}\;\;Q:=\mathbb{R}^{k}\backslash V.
\end{equation}
Thus, $H^{\mu}(V)$ is a separable Hilbert space.
The norm \eqref{8f40} is induced by the inner product
$$
(u_{1},u_{2})_{H^{\mu}(V)}:= (w_{1}-\Upsilon
w_{1},w_{2}-\Upsilon w_{2})_{H^{\mu}(\mathbb{R}^{k})},
$$
where $w_{j}\in H^{\mu}(\mathbb{R}^{k})$, $w_{j}=u_{j}$ in $V$
for each $j\in\{1,\,2\}$, and $\Upsilon$ is the orthogonal projector of the space $H^{\mu}(\mathbb{R}^{k})$ onto its subspace \eqref{8f41}. The spaces $H^{\mu}(V)$ and $H^{\mu}_{Q}(\mathbb{R}^{k})$ were introduced and investigated by Volevich and Paneah
\cite[Section~3]{VolevichPaneah65}.

It follows directly from the definition of $H^{\mu}(V)$ and properties of $H^{\mu}(\mathbb{R}^{k})$ that the space $H^{\mu}(V)$ is continuously embedded in the linear topological space $\mathcal{D}'(V)$ of all distributions on $V$ and that the set
$$
C^{\infty}_{0}(\overline{V}):=\bigl\{w\!\upharpoonright\!\overline{V}:\, w\in C^{\infty}_{0}(\mathbb{R}^{k})\bigr\}
$$
is dense in $H^{\mu}(V)$.

Suppose that the integer $k\geq2$. Dealing with the above-stated parabolic problems, we need the H\"ormander spaces $H^{\mu}(\mathbb{R}^{k})$ and their versions in the case where the regularity index $\mu$ takes the form
\begin{equation}\label{8f4}
\begin{gathered}
\mu(\xi',\xi_{k})=\bigl(1+|\xi'|^2+|\xi_{k}|\bigr)^{s/2}
\varphi\bigl((1+|\xi'|^2+|\xi_{k}|)^{1/2}\bigr)\\
\mbox{for all}\;\;\xi'\in\mathbb{R}^{k-1}\;\;\mbox{and}\;\;
\xi_{k}\in\mathbb{R}.
\end{gathered}
\end{equation}
Here, the number parameter $s$ is real, whereas the function parameter $\varphi$ runs over a certain class~$\mathcal{M}$.

By definition, the class $\mathcal{M}$ consists of all Borel measurable functions $\varphi:[1,\infty)\rightarrow(0,\infty)$ such that
\begin{itemize}
  \item [a)] both the functions $\varphi$ and $1/\varphi$ are bounded on each compact interval $[1,b]$, with $1<b<\infty$;
  \item [b)] the function $\varphi$ varies slowly at infinity in the sense of Karamata \cite{Karamata30a}, i.e.
      $\varphi(\lambda r)/\varphi(r)\rightarrow\nobreak1$ as $r\rightarrow\infty$ for each $\lambda>0$.
\end{itemize}

The theory of slowly varying functions (at infinity) is expounded, e.g., in \cite{BinghamGoldieTeugels89, Seneta76}. Their standard  examples are the functions
\begin{equation*}
\varphi(r):=(\log r)^{\theta_{1}}\,(\log\log r)^{\theta_{2}} \ldots
(\,\underbrace{\log\ldots\log}_{k\;\mbox{\tiny{times}}}r\,)^{\theta_{k}}
\quad\mbox{of}\;\;r\gg1,
\end{equation*}
where the parameters $k\in\mathbb{N}$ and
$\theta_{1},\theta_{2},\ldots,\theta_{k}\in\mathbb{R}$ are arbitrary.

Let $s\in\mathbb{R}$ and $\varphi\in\mathcal{M}$. We put $H^{s,s/2;\varphi}(\mathbb{R}^{k}):=H^{\mu}(\mathbb{R}^{k})$ in the case where $\mu$ is of the form~\eqref{8f4}. Specifically, if $\varphi(r)\equiv1$, then $H^{s,s/2;\varphi}(\mathbb{R}^{k})$ becomes the anisotropic Sobolev inner product space $H^{s,s/2}(\mathbb{R}^{k})$ of order $(s,s/2)$. Generally, if $\varphi\in\mathcal{M}$ is arbitrary, then the following continuous and dense embeddings hold:
\begin{equation}\label{8f5}
H^{s_{1},s_{1}/2}(\mathbb{R}^{k})\hookrightarrow
H^{s,s/2;\varphi}(\mathbb{R}^{k})\hookrightarrow
H^{s_{0},s_{0}/2}(\mathbb{R}^{k})\quad\mbox{whenever}\quad s_{0}<s<s_{1}.
\end{equation}
Indeed, let $s_{0}<s<s_{1}$; since $\varphi\in\mathcal{M}$, there exist positive numbers $c_0$ and $c_1$ such that $c_0\,r^{s_0-s}\leq\varphi(r)\leq c_1\,r^{s_1-s}$ for every $r\geq1$ (see e.g., \cite[Section 1.5, Property $1^\circ$]{Seneta76}).
Then
\begin{align*}
c_{0}\bigl(1+|\xi'|^2+|\xi_{k}|\bigr)^{s_{0}/2}&\leq
\bigl(1+|\xi'|^2+|\xi_{k}|\bigr)^{s/2}
\varphi\bigl((1+|\xi'|^2+|\xi_{k}|)^{1/2}\bigr)\\
&\leq c_{1}\bigl(1+|\xi'|^2+|\xi_{k}|\bigr)^{s_{1}/2}
\end{align*}
for arbitrary $\xi'\in\mathbb{R}^{k-1}$ and $\xi_{k}\in\mathbb{R}$.
This directly entails the continuous embeddings~\eqref{8f5}.
They are dense because the set $C^{\infty}_{0}(\mathbb{R}^{k})$
is dense in all the spaces from \eqref{8f5}.

Consider the class of H\"ormander inner product spaces
\begin{equation}\label{8f6}
\bigl\{H^{s,s/2;\varphi}(\mathbb{R}^{k}):\,
s\in\mathbb{R},\,\varphi\in\mathcal{M}\,\bigr\}.
\end{equation}
The embeddings \eqref{8f5} show, that in \eqref{8f6} the function parameter $\varphi$ defines additional regularity with respect to the  basic anisotropic $(s,s/2)$-regularity. Specifically, if $\varphi(r)\rightarrow\infty$ [or $\varphi(r)\rightarrow\nobreak0$] as $r\rightarrow\infty$, then $\varphi$ defines additional positive [or negative] regularity. In other words, $\varphi$ refines the basic smoothness $(s,s/2)$.

We need versions of the function spaces \eqref{8f6} for the cylinder $\Omega=G\times(0,\tau)$ and its lateral boundary  $S=\Gamma\times(0,\tau)$. We put $H^{s,s/2;\varphi}(\Omega):=H^{\mu}(\Omega)$ in the case where $\mu$ is of the form~\eqref{8f4} with $k:=n+1$. For the function space $H^{s,s/2;\varphi}(\Omega)$, the numbers $s$ and $s/2$ serve as the regularity indices of distributions $u(x,t)$ with respect to the spatial variable $x\in G$ and to the time variable $t\in(0,\tau)$ respectively.

Following \cite[Section~1]{Los16JMathSci}, we will define the function space $H^{s,s/2;\varphi}(S)$ with the help of special local charts on~$S$. Let $s>0$ and $\varphi\in\mathcal{M}$. We put $H^{s,s/2;\varphi}(\Pi):=H^{\mu}(\Pi)$ for the strip $\Pi:=\mathbb{R}^{n-1}\times(0,\tau)$ in the case where $\mu$ is defined by formula \eqref{8f4} with $k:=n$. Recall that, according to our assumption $\Gamma=\partial\Omega$ is an infinitely smooth closed manifold of dimension $n-1$, the $C^{\infty}$-structure on $\Gamma$ being induced by $\mathbb{R}^{n}$. From this structure we arbitrarily choose a finite atlas formed by local charts $\nobreak{\theta_{j}:\mathbb{R}^{n-1}\leftrightarrow \Gamma_{j}}$ with $j=1,\ldots,\lambda$. Here, the open sets $\Gamma_{1},\ldots,\Gamma_{\lambda}$
make up a covering of~$\Gamma$. We also arbitrarily choose  functions $\chi_{j}\in C^{\infty}(\Gamma)$, with $j=1,\ldots,\lambda$, so that $\mathrm{supp}\,\chi_{j}\subset\Gamma_{j}$ and $\chi_{1}+\cdots\chi_{\lambda}=1$ on $\Gamma$.

By definition, the linear space $H^{s,s/2;\varphi}(S)$ consists of all
square integrable functions $\nobreak{g:S\to\mathbb{C}}$ that the function
$$
g_{j}(x,t):=\chi_{j}(\theta_{j}(x))\,g(\theta_{j}(x),t)
\quad\mbox{of}\;\;x\in\mathbb{R}^{n-1}\;\;\mbox{and}\;\;t\in(0,\tau)
$$
belongs to $H^{s,s/2;\varphi}(\Pi)$ for each number
$j\in\{1,\ldots,\lambda\}$. The inner product in $H^{s,s/2;\varphi}(S)$ is defined by the formula
\begin{equation*}
(g,g')_{H^{s,s/2;\varphi}(S)}:=\sum_{j=1}^{\lambda}\,
(g_{j},g'_{j})_{H^{s,s/2;\varphi}(\Pi)},
\end{equation*}
where $g,g'\in H^{s,s/2;\varphi}(S)$. This inner product naturally induces the norm
$$
\|g\|_{H^{s,s/2;\varphi}(S)}:=(g,g)^{1/2}_{H^{s,s/2;\varphi}(S)}.
$$
The space $H^{s,s/2;\varphi}(S)$ is complete (i.~e. Hilbert) and does not depend up to equivalence of norms on the choice of local charts and partition of unity on $\Gamma$ \cite[Theorem~1]{Los16JMathSci}. Note that this space is actually defined with the help of the following special local charts on $S$:
\begin{equation}\label{8f-local}
\theta_{j}^*:\Pi=\mathbb{R}^{n-1}\times(0,\tau)\leftrightarrow
\Gamma_{j}\times(0,\tau),\quad j=1,\ldots,\lambda,
\end{equation}
where $\theta_{j}^*(x,t):=(\theta_{j}(x),t)$ for all
$x\in\mathbb{R}^{n-1}$ and $t\in(0,\tau)$.

We also need isotropic H\"ormander spaces $H^{s;\varphi}(V)$ over an arbitrary open nonempty set $V\subseteq\mathbb{R}^{k}$ with $k\geq1$. Let $s\in\mathbb{R}$ and $\varphi\in\mathcal{M}$. We put $H^{s;\varphi}(V):=H^{\mu}(V)$ in the case where the regularity index $\mu$ takes the form
\begin{equation}\label{8f50}
\mu(\xi)=\bigl(1+|\xi|^2\bigr)^{s/2}\varphi\bigl((1+|\xi|^2)^{1/2}\bigr)
\quad\mbox{for arbitrary}\;\;\xi\in\mathbb{R}^{k}.
\end{equation}
Since the function \eqref{8f50} is radial (i.e., depends only on
$|\xi|$), the space $H^{s;\varphi}(V)$ is isotropic.
We will use the spaces $H^{s;\varphi}(V)$ given over the whole Euclidean space $V:=\mathbb{R}^{k}$ or over the domain $V:=G$ in $\mathbb{R}^{n}$.

Besides, we will use H\"ormander spaces $H^{s;\varphi}(\Gamma)$ over $\Gamma=\partial\Omega$. The are defined with the help of the above-mentioned collection of local charts $\{\theta_{j}\}$ and partition of unity $\{\chi_{j}\}$ on $\Gamma$ similarly to the spaces over $S$. Let $s\in\mathbb{R}$ and $\varphi\in\mathcal{M}$. By definition, the linear space  $H^{s;\varphi}(\Gamma)$ consists of all distributions $\omega\in\mathcal{D}'(\Gamma)$ on $\Gamma$ that for each number $j\in\{1,\ldots,\lambda\}$ the distribution
$\omega_{j}(x):=\chi_{j}(\theta_{j}(x))\,\omega(\theta_{j}(x))$ of
$x\in\mathbb{R}^{n-1}$ belongs to $H^{s;\varphi}(\mathbb{R}^{n-1})$.
The inner product in $H^{s;\varphi}(\Gamma)$ is defined by the formula
\begin{equation*}
(\omega,\omega')_{H^{s;\varphi}(\Gamma)}:=
\sum_{j=1}^{\lambda}\,
(\omega_{j},\omega'_{j})_{H^{s;\varphi}(\mathbb{R}^{n-1})},
\end{equation*}
where $\omega,\omega'\in H^{s;\varphi}(\Gamma)$. It induces the norm
$$
\|\omega\|_{H^{s;\varphi}(\Gamma)}:=
(\omega,\omega)^{1/2}_{H^{s;\varphi}(\Gamma)}.
$$
The space $H^{s;\varphi}(\Gamma)$ is Hilbert separable and does not depend up to equivalence of norms on our choice of local charts and partition of unity on $\Gamma$ \cite[Theorem 3.6(i)]{MikhailetsMurach08MFAT1}.

Note that the classes of isotropic inner product spaces
\begin{equation*}
\bigl\{H^{s;\varphi}(V):s\in\mathbb{R},\;\varphi\in\mathcal{M}\bigr\}
\quad\mbox{and}\quad
\bigl\{H^{s;\varphi}(\Gamma):s\in\mathbb{R},\;\varphi\in\mathcal{M}\bigr\}
\end{equation*}
were selected, investigated, and systematically applied to elliptic differential operators and elliptic boundary-value problems by Mikhailets and Murach \cite{MikhailetsMurach14, MikhailetsMurach12BJMA2}.

If $\varphi\equiv1$, then the considered spaces $H^{s,s/2;\varphi}(\cdot)$ and $H^{s;\varphi}(\cdot)$ become the Sobolev spaces $H^{s,s/2}(\cdot)$ and $H^{s}(\cdot)$ respectively. It follows directly from \eqref{8f5} that
\begin{equation}\label{8f5a}
H^{s_{1},s_{1}/2}(\cdot)\hookrightarrow
H^{s,s/2;\varphi}(\cdot)\hookrightarrow
H^{s_{0},s_{0}/2}(\cdot)\quad\mbox{whenever}\quad s_{0}<s<s_{1}.
\end{equation}
Analogously,
\begin{equation}\label{8f5b}
H^{s_{1}}(\cdot)\hookrightarrow
H^{s;\varphi}(\cdot)\hookrightarrow
H^{s_{0}}(\cdot)\quad\mbox{whenever}\quad s_{0}<s<s_{1};
\end{equation}
see \cite[Theorems 2.3(iii) and 3.3(iii)]{MikhailetsMurach14}. These embeddings are continuous and dense. Of course, if $s=0$, then $H^{s}(\cdot)=H^{s,s/2}(\cdot)$ is the Hilbert space $L_2(\cdot)$ of all square integrable functions given on the corresponding measurable set.

In the Sobolev case of $\varphi\equiv1$, we will omit the index $\varphi$ in designations of function spaces that will be introduced on the base of the H\"ormander spaces $H^{s,s/2;\varphi}(\cdot)$ and $H^{s;\varphi}(\cdot)$.

\section{Main results}\label{8sec4}

Consider first the parabolic problem \eqref{8f1}--\eqref{8f2}, which corresponds to the Dirichlet boundary condition on~$S$. In order that a regular enough solution $u$ to this problem exist, the right-hand sides of the problem should satisfy certain compatibility conditions (see, e.g., \cite[Section~11]{AgranovichVishik64} or \cite[Chapter~4, Section~5]{LadyzhenskajaSolonnikovUraltzeva67}). These conditions consist in that the partial derivatives $\partial^k_t u(x,t)\big|_{t=0}$, which could be found from the
parabolic equation \eqref{8f1} and initial condition \eqref{8f3}, should satisfy the boundary condition \eqref{8f2} and some relations that are obtained by means of the differentiation of the boundary condition with respect to~$t$. To write these compatibility conditions we use Sobolev inner product spaces.

We associate the linear mapping
\begin{gather}\label{8f7}
\Lambda_0:\,u\mapsto\bigl(Au,u\!\upharpoonright\!\overline{S}, u(\cdot,0)\bigr),\quad \mbox{where}\quad u\in
C^{\infty}(\overline{\Omega}),
\end{gather}
with the problem \eqref{8f1}--\eqref{8f2}. Let real $s\geq2$; the mapping \eqref{8f7} extends uniquely (by continuity) to a bounded linear operator
\begin{equation}\label{8f7a}
\Lambda_0:\,H^{s,s/2}(\Omega)\rightarrow
H^{s-2,s/2-1}(\Omega)\oplus
H^{s-1/2,s/2-1/4}(S)\oplus H^{s-1}(G).
\end{equation}
This follows directly from \cite[Chapter~I, Lemma~4, and Chapter~II, Theorems 3 and 7]{Slobodetskii58}. Choosing any function $u(x,t)$ from the space $H^{s,s/2}(\Omega)$, we define the right-hand sides
\begin{equation}\label{8f7b}
f\in H^{s-2,s/2-1}(\Omega),\quad g\in H^{s-1/2,s/2-1/4}(S),\quad
\mbox{and}\quad h\in H^{s-1}(G)
\end{equation}
of the problem by the formula $(f,g,h):=\Lambda_0u$ with the help of this bounded operator.

According to \cite[Chapter~II, Theorem 7]{Slobodetskii58}, the traces
$\partial^{\,k}_t u(\cdot,0)\in H^{s-1-2k}(G)$ are well defined by closure for all $k\in\mathbb{Z}$ such that $0\leq k<s/2-1/2$ (and only for these $k$). Using \eqref{8f1} and \eqref{8f3}, we express these traces in terms of the functions $f(x,t)$ and $h(x)$ by the recurrent formula
\begin{equation}\label{8f9}
\begin{aligned}
u(x,0)&=h(x),\\
\partial^{k}_t u(x,0)&=
-\sum\limits_{|\alpha|\leq2}\sum\limits_{q=0}^{k-1}
\binom{k-1}{q}\partial^{k-1-q}_t a_{\alpha}(x,0)\,
D^\alpha_x\partial^{q}_t u(x,0)
+\partial^{k-1}_t f(x,0)\\
&\mbox{for each}\quad k\in\mathbb{Z}\quad\mbox{such that}\quad
1\leq k<s/2-1/2,
\end{aligned}
\end{equation}
the equalities holding for almost all $x\in G$.

Besides, the traces $\partial^{k}_t g(\cdot,0)\in H^{s-3/2-2k}(\Gamma)$ are well defined by closure for all $k\in\mathbb{Z}$ such that $0\leq k<s/2-3/4$ (and only for these $k$). Therefore, owing to the Dirichlet boundary condition \eqref{8f2}, the equality
\begin{equation}\label{8f68}
\partial^{k}_t g(x,0)=\partial^{k}_t u(x,0)
\quad\mbox{for almost all}\quad x\in\Gamma
\end{equation}
holds for these integers $k$. The right-hand part of this equality is well defined because the function $\partial^{k}_t u(\cdot,0)\in H^{s-1-2k}(G)$ has the trace $\partial^{k}_t u(\cdot,0)\!\upharpoonright\!\Gamma\in H^{s-3/2-2k}(\Gamma)$ in view of $s-3/2-2k>0$.

Now, substituting \eqref{8f9} into \eqref{8f68}, we obtain the compatibility conditions
\begin{equation}\label{8f10}
\partial^{k}_t g\!\upharpoonright\!\Gamma=v_k\!\upharpoonright\!\Gamma,
\quad\mbox{with}\;\;\;k\in\mathbb{Z}\;\;\;\mbox{and}\;\;\;
0\leq k<s/2-3/4.
\end{equation}
Here, the functions $v_k$ are defined by the recurrent formula
\begin{equation}\label{8f69}
\begin{aligned}
v_0(x)&:=h(x),\\
v_k(x)&:=-\sum\limits_{|\alpha|\leq2}\sum\limits_{q=0}^{k-1}
\binom{k-1}{q}\partial^{k-1-q}_t a_{\alpha}(x,0)\,
D^\alpha_x v_{q}(x)+\partial^{k-1}_t f(x,0)\\
&\mbox{for each}\quad k\in\mathbb{Z}\quad\mbox{such that}\quad
1\leq k<s/2-1/2,
\end{aligned}
\end{equation}
these relations holding for almost all $x\in G$. Since
\begin{equation}\label{8f69aa}
v_k\in H^{s-1-2k}(G)\quad\mbox{for each}\quad
k\in\mathbb{Z}\cap[0,s/2-1/2)
\end{equation}
due to \eqref{8f7b}, the trace $v_k\!\upharpoonright\!\Gamma\in H^{s-3/2-2k}(\Gamma)$ is defined by closure whenever $s-3/2-2k>0$. Thus, the compatibility conditions \eqref{8f10} are well posed.

For instance, if $2<s\leq7/2$, then formula \eqref{8f10} gives  one compatibility condition $\nobreak{g\!\upharpoonright\!\Gamma=h\!\upharpoonright\!\Gamma}$. Next, if $7/2<s\leq11/2$, then \eqref{8f10} gives two compatibility conditions
$g\!\upharpoonright\!\Gamma=h\!\upharpoonright\!\Gamma$ and
$$
\partial_t g\!\upharpoonright\!\Gamma=
\biggl(-\sum\limits_{|\alpha|\leq2}a_{\alpha}(x,0)\,
D^\alpha_{x}h(x)+f(x,0)\biggr)\!\upharpoonright\!\Gamma,
$$
and so on.

We put $E_{0}:=\{2r+3/2:1\leq r\in\mathbb{Z}\}$. Note that $E_{0}$ is the set of all discontinuities of the function that assigns the number of compatibility conditions \eqref{8f10} to $s\geq2$.

Our main result on the parabolic problem \eqref{8f1}--\eqref{8f2} consists in that the linear mapping \eqref{8f7} extends uniquely  to an isomorphism between appropriate pairs of H\"ormander spaces introduced in the previous section. Let us indicate these spaces. We arbitrarily choose a real number $s>2$ and function parameter $\varphi\in\mathcal{M}$. We take $H^{s,s/2;\varphi}(\Omega)$ as the source space of this isomorphism; otherwise speaking, $H^{s,s/2;\varphi}(\Omega)$ serves as a space of solutions $u$ to the problem. To introduce the target space of the isomorphism, consider the Hilbert space
\begin{gather*}
\mathcal{H}_{0}^{s-2,s/2-1;\varphi}:=
H^{s-2,s/2-1;\varphi}(\Omega)\oplus
H^{s-1/2,s/2-1/4;\varphi}(S)\oplus H^{s-1;\varphi}(G).
\end{gather*}
In the Sobolev case of $\varphi\equiv1$ this space coincides with the target space of the bounded operator \eqref{8f7a}. The target space of the isomorphism is imbedded in $\mathcal{H}_{0}^{s-2,s/2-1;\varphi}$ and is denoted by  $\mathcal{Q}_{0}^{s-2,s/2-1;\varphi}$. We separately define this space in the $s\notin E_{0}$ case and $s\in E_{0}$ case.

Suppose first that $s\notin E_{0}$. By definition, the linear space $\mathcal{Q}_{0}^{s-2,s/2-1;\varphi}$ consists of all vectors $\bigl(f,g,h\bigr)\in\mathcal{H}_{0}^{s-2,s/2-1;\varphi}$ that satisfy the compatibility conditions \eqref{8f10}. As we have noted, these conditions are well defined for every $\bigl(f,g,h\bigr)\in \mathcal{H}_{0}^{s-2-\varepsilon,s/2-1-\varepsilon/2}$ for sufficiently small $\varepsilon>0$. Hence, they are also well defined for every $\bigl(f,g,h\bigr)\in\mathcal{H}_{0}^{s-2,s/2-1;\varphi}$ due to the continuous embedding
\begin{equation}\label{8f69a}
\mathcal{H}_{0}^{s-2,s/2-1;\varphi}
\hookrightarrow \mathcal{H}_{0}^{s-2-\varepsilon,s/2-1-\varepsilon/2}.
\end{equation}
The latter follows directly from \eqref{8f5a} and \eqref{8f5b}. Thus, our definition is reasonable.

We endow the linear space $\mathcal{Q}_{0}^{s-2,s/2-1;\varphi}$ with the inner product and norm in the Hilbert space
$\mathcal{H}_{0}^{s-2,s/2-1;\varphi}$. The space $\mathcal{Q}_{0}^{s-2,s/2-1;\varphi}$
is complete, i.e. a Hilbert one. Indeed, if the number $\varepsilon>0$ is sufficiently small, then
$$
\mathcal{Q}_{0}^{s-2,s/2-1;\varphi}=
\mathcal{H}_{0}^{s-2,s/2-1;\varphi}\cap
\mathcal{Q}_{0}^{s-2-\varepsilon,s/2-1-\varepsilon/2}.
$$
Here, the space $\mathcal{Q}_{0}^{s-2-\varepsilon,s/2-1-\varepsilon/2}$ is complete because the differential operators and traces operators used in the compatibility conditions are bounded on the corresponding pairs of Sobolev spaces. Therefore the right-hand side of this equality is complete with respect to the sum of the norms in the components of the intersection, this sum being equivalent to the norm in $\mathcal{H}_{0}^{s-2,s/2-1;\varphi}$ due to \eqref{8f69a}. Thus, the space $\mathcal{Q}_{0}^{s-2,s/2-1;\varphi}$ is complete (with respect to the latter norm).

If $s\in E_{0}$, then we define the Hilbert space $\mathcal{Q}_{0}^{s-2,s/2-1;\varphi}$
by means of the interpolation between its analogs just introduced. Namely, we put
\begin{equation}\label{8f71}
\mathcal{Q}_{0}^{s-2,s/2-1;\varphi}:=\bigl[
\mathcal{Q}_{0}^{s-2-\varepsilon,s/2-1-\varepsilon/2;\varphi},
\mathcal{Q}_{0}^{s-2+\varepsilon,s/2-1+\varepsilon/2;\varphi}
\bigr]_{1/2}.
\end{equation}
Here, the number $\varepsilon\in(0,1/2)$ is arbitrarily chosen, and the right-hand side of the equality is the result of the interpolation
of the written pair of Hilbert spaces with the parameter~$1/2$. We will recall the definition of the interpolation between Hilbert spaces in  Section~\ref{8sec5}. The Hilbert space $\mathcal{Q}_{0}^{s-2,s/2-1;\varphi}$ defined by formula \eqref{8f71} does not depend on the choice of $\varepsilon$ up to equivalence of norms and is continuously embedded in $\mathcal{H}_{0}^{s-2,s/2-1;\varphi}$. This will be shown in Remark~\ref{8rem1} at the end of Section~\ref{8sec6}.

Now we can formulate our main result concerning the parabolic initial-boundary value problem \eqref{8f1}--\eqref{8f2}.

\begin{theorem}\label{8th1}
For arbitrary $s>2$ and $\varphi\in\nobreak\mathcal{M}$ the mapping \eqref{8f7} extends uniquely (by continuity) to an isomorphism
\begin{equation}\label{8f8}
\Lambda_{0}:H^{s,s/2;\varphi}(\Omega)\leftrightarrow
\mathcal{Q}_{0}^{s-2,s/2-1;\varphi}.
\end{equation}
\end{theorem}

Otherwise speaking, the parabolic problem \eqref{8f1}--\eqref{8f2} is well posed (in the sense of Hadamard) on the pair of Hilbert spaces $H^{s,s/2;\varphi}(\Omega)$ and $\mathcal{Q}_{0}^{s-2,s/2-1;\varphi}$ whenever $s>2$ and $\varphi\in\nobreak\mathcal{M}$, the right-hand side $\bigl(f,g,h\bigr)\in\mathcal{Q}_{0}^{s-2,s/2-1;\varphi}$ of the problem being defined by closer for an arbitrary function $u\in H^{s,s/2;\varphi}(\Omega)$.

Note that the necessity to define the target space $\mathcal{Q}_{0}^{s-2,s/2-1;\varphi}$ separately in the  $s\in E_{0}$ case is caused by the following: if we defined this space for $s\in E_{0}$ in the way used in the $s\notin E_{0}$ case, then the isomorphism \eqref{8f8} would not hold at least for $\varphi\equiv1$. This follows from a result by Solonnikov \cite[Section~6]{Solonnikov64}.

Consider now the parabolic problem \eqref{8f1}, \eqref{8f3}, \eqref{8f2n}, which corresponds to the first order boundary condition on $S$. Let us write the compatibility conditions for the right-hand sides of this problem.

We associate the linear mapping
\begin{gather}\label{8f11}
\Lambda_1:\,u\mapsto\bigl(Au,Bu,u(\cdot,0)\bigr),\quad
\mbox{where}\quad u\in C^{\infty}(\overline{\Omega}),
\end{gather}
with the problem \eqref{8f1}, \eqref{8f3}, \eqref{8f2n}. For arbitrary  real $s\geq2$, this mapping extends uniquely (by continuity) to a bounded linear operator
\begin{equation}\label{8f7N}
\Lambda_1:\,H^{s,s/2}(\Omega)\rightarrow
H^{s-2,s/2-1}(\Omega)\oplus
H^{s-3/2,s/2-3/4}(S)\oplus H^{s-1}(G).
\end{equation}
Choosing any function $u(x,t)$ from $H^{s,s/2}(\Omega)$, we define the right-hand sides
\begin{equation*}
f\in H^{s-2,s/2-1}(\Omega),\quad g\in H^{s-3/2,s/2-3/4}(S),\quad
\mbox{and}\quad h\in H^{s-1}(G)
\end{equation*}
of the problem by the formula $(f,g,h):=\Lambda_1u$ with the help of this bounded operator. Here, unlike \eqref{8f7b}, the inclusion $u\in H^{s,s/2}(\Omega)$ implies $g=Bu\in H^{s-3/2,s/2-3/4}(S)$ due to \cite[Chapter~II, Theorem 7]{Slobodetskii58}. According to this theorem, the traces $\partial^{k}_t g(\cdot,0)\in H^{s-5/2-2k}(\Gamma)$ are defined by closure for all $k\in\mathbb{Z}$ such that $0\leq k<s/2-5/4$ (and only for these~$k$). We can express these traces in terms of the function $u(x,t)$ and its time derivatives; namely,
\begin{equation}\label{8f7bb}
\begin{aligned}
\partial^{k}_t g(x,0)&=\bigl(\partial^{k}_{t}Bu(x,t)\bigr)|_{t=0}\\
&=\sum_{q=0}^{k}\binom{k}{q}\biggl(\,
\sum_{j=1}^{n}\partial^{k-q}_{t}b_j(x,0)\,D_j\partial^{q}_{t}u(x,0)+
\partial^{k-q}_t b_0(x,0)\,\partial^{\,q}_t u(x,0)\biggr)
\end{aligned}
\end{equation}
for almost all $x\in\Gamma$. Here, all the functions $u(x,0)$, $\partial_{t}u(x,0)$,..., $\partial^{k}_{t}u(x,0)$ of $x\in G$ are expressed in terms of the functions $f(x,t)$ and $h(x)$ by the recurrent formula \eqref{8f9}.

Substituting \eqref{8f9} in the right-hand side of formula \eqref{8f7bb}, we obtain the compatibility conditions
\begin{equation}\label{8f14}
\partial^{k}_t g\!\upharpoonright\!\Gamma=B_k[v_0,\dots,v_k],
\quad\mbox{with}\;\;\;k\in\mathbb{Z}\;\;\;\mbox{and}\;\;\;
0\leq k<s/2-5/4.
\end{equation}
Here, the functions $v_0$, $v_1$,..., $v_k$ are defined on $G$ by the recurrent formula \eqref{8f69}, and we put
\begin{equation*}
B_k[v_0,\dots,v_k](x):=\sum_{q=0}^{k}\binom{k}{q}\biggl(\,
\sum_{j=1}^{n}\partial^{\,k-q}_t b_j(x,0)\,D_jv_q(x)+
\partial^{\,k-q}_t b_0(x,0)\,v_q(x)\biggr)
\end{equation*}
for almost all $x\in\Gamma$. Here, we consider the functions $D_jv_q(x)$ and $v_q(x)$ of $x\in\Gamma$ as the traces of the functions $D_jv_q\in H^{s-2-2q}(G)$ and $v_q\in H^{s-1-2q}(G)$ on $\Gamma$. Thus,
$$
B_k[v_0,\dots,v_k]\in H^{s-5/2-2k}(\Gamma)
$$
due to \eqref{8f69aa} and $s-5/2-2k>0$. Note that if $s\leq5/2$, then there are no compatibility conditions.

We set $E_{1}:=\{2r+1/2:1\leq r\in\mathbb{Z}\}$. Observe that $E_{1}$ is the set of all discontinuities of the function that assigns the number of compatibility conditions \eqref{8f14} to $s\geq2$.

To formulate our isomorphism theorem for the parabolic problem \eqref{8f1}, \eqref{8f3}, \eqref{8f2n}, we introduce the source and target spaces of this isomorphism. Let $s>2$ and $\varphi\in\mathcal{M}$. As above, we take $H^{s,s/2;\varphi}(\Omega)$ as the source space. The target space denoted by $\mathcal{Q}_{1}^{s-2,s/2-1;\varphi}$ is embedded in the Hilbert space
\begin{gather*}
\mathcal{H}_{1}^{s-2,s/2-1;\varphi}:=
H^{s-2,s/2-1;\varphi}(\Omega)\oplus
H^{s-3/2,s/2-3/4;\varphi}(S)\oplus H^{s-1;\varphi}(G).
\end{gather*}
In the Sobolev case of $\varphi\equiv1$ this space coincides with the target space of the bounded operator \eqref{8f7N}.

If $s\notin E_{1}$, then the linear space $\mathcal{Q}_{1}^{s-2,s/2-1;\varphi}$ is defined to consist of all vectors $\bigl(f,g,h\bigr)\in \mathcal{H}_{1}^{s-2,s/2-1;\varphi}$ that satisfy the compatibility conditions \eqref{8f14}. The definition is reasonable because these conditions are well defined for every $\bigl(f,g,h\bigr)\in \mathcal{H}_{1}^{s-2-\varepsilon,s/2-1-\varepsilon/2}$ for sufficiently small $\varepsilon>0$ and because
\begin{equation}\label{8f69N}
\mathcal{H}_{1}^{s-2,s/2-1;\varphi}
\hookrightarrow \mathcal{H}_{1}^{s-2-\varepsilon,s/2-1-\varepsilon/2}.
\end{equation}
This continuous embedding follows immediately from \eqref{8f5a} and \eqref{8f5b}. The linear space $\mathcal{Q}_{1}^{s-2,s/2-1;\varphi}$ is
endowed with the inner product and the norm in the Hilbert space
$\mathcal{H}_{1}^{s-2,s/2-1;\varphi}$. The space $\mathcal{Q}_{1}^{s-2,s/2-1;\varphi}$ is
complete, i.e. a Hilbert one. This is justified by the same reasoning as we have used to prove the completeness of $\mathcal{Q}_{0}^{s-2,s/2-1;\varphi}$. Note that if $2<s<5/2$, then the spaces $\mathcal{H}_{1}^{s-2,s/2-1;\varphi}$ and $\mathcal{Q}_{1}^{s-2,s/2-1;\varphi}$ coincide because the compatibility conditions \eqref{8f14} are absent.

If $s\in E_{1}$, then we define the Hilbert space $\mathcal{Q}_{1}^{s-2,s/2-1;\varphi}$ by the interpolation, namely
\begin{equation}\label{8f72}
\mathcal{Q}_{1}^{s-2,s/2-1;\varphi}:=\bigl[
\mathcal{Q}_{1}^{s-2-\varepsilon,s/2-1-\varepsilon/2;\varphi},
\mathcal{Q}_{1}^{s-2+\varepsilon,s/2-1+\varepsilon/2;\varphi}
\bigr]_{1/2},
\end{equation}
with the number $\varepsilon\in(0,1/2)$  chosen arbitrarily. This Hilbert space does not depend on the choice of $\varepsilon$ up to equivalence of norms and is embedded continuously in $\mathcal{H}_{1}^{s-2,s/2-1;\varphi}$, which will be shown in Remark~\ref{8rem1}.

Now we can formulate our main result concerning the parabolic initial-boundary value problem \eqref{8f1}, \eqref{8f3}, \eqref{8f2n}.

\begin{theorem}\label{8th2}
For arbitrary $s>2$ and $\varphi\in\nobreak\mathcal{M}$ the mapping \eqref{8f11} extends uniquely (by continuity) to an isomorphism
\begin{equation}\label{8f12}
\Lambda_{1}:H^{s,s/2;\varphi}(\Omega)\leftrightarrow
\mathcal{Q}_{1}^{s-2,s/2-1;\varphi}.
\end{equation}
\end{theorem}

In other words, the parabolic problem \eqref{8f1}, \eqref{8f3}, \eqref{8f2n} is well posed on the pair of Hilbert spaces $H^{s,s/2;\varphi}(\Omega)$ and $\mathcal{Q}_{1}^{s-2,s/2-1;\varphi}$ whenever $s>2$ and $\varphi\in\nobreak\mathcal{M}$, the right-hand side  $\bigl(f,g,h\bigr)\in\mathcal{Q}_{1}^{s-2,s/2-1;\varphi}$ of this problem being defined by closer for an arbitrary function $u\in H^{s,s/2;\varphi}(\Omega)$.

Note that the necessity to define the target space $\mathcal{Q}_{1}^{s-2,s/2-1;\varphi}$ separately in the  $s\in E_{1}$ case is stipulated by a similar cause as that indicated for the space $\mathcal{Q}_{0}^{s-2,s/2-1;\varphi}$. Namely, if we defined this space for $s\in E_{1}$ in the way used in the $s\notin E_{1}$ case, then the isomorphism \eqref{8f12} would not hold at least when $\varphi\equiv1$ and \eqref{8f2n} is the Neumann boundary condition (see \cite[Section~6]{Solonnikov64}).

Theorems \ref{8th1} and \ref{8th2} are known in the Sobolev case where $\varphi\equiv1$ and neither $s$ nor $s/2$ is half-integer.
Namely, they are contained in Agranovich and Vishik's result
\cite[Theorem~12.1]{AgranovichVishik64} in the case of  $s,s/2\in\mathbb{Z}$ and are covered by Lions and Magenes' result \cite[Theorem~6.2]{LionsMagenes72ii}. Solonnikov \cite[Theorem~17]{Solonnikov64} proved the corresponding a priory estimates for anisotropic Sobolev norms of solutions to the problem \eqref{8f1}--\eqref{8f2} and to the problem \eqref{8f1}, \eqref{8f3}, \eqref{8f2n} provided that \eqref{8f2n} is the Neumann boundary condition. Note that these results include the limiting case of $s=2$.

In Section \ref{8sec6} we will deduce Theorems \ref{8th1} and \ref{8th2} from the above-mentioned results with the help of the method of interpolation with a function parameter between Hilbert spaces, specifically between Sobolev inner product spaces. Therefore we devote the next section to this method and its applications to Sobolev and H\"ormander spaces.

\section{Interpolation with a function parameter between Hilbert spaces}\label{8sec5}

This method of interpolation is a natural generalization of the classical interpolation method by S.~Krein and J.-L.~Lions to the case when a general enough function is used instead of a number as an interpolation parameter; see, e.g., monographs \cite[Chapter~IV, Section~1, Subsection~10]{KreinPetuninSemenov82} and \cite[Chapter~1, Sections 2 and 5]{LionsMagenes72i}. For our purposes, it is sufficient to restrict the discussion of the interpolation with a function parameter to the case of separable complex Hilbert spaces. We mainly follow the monograph \cite[Section~1.1]{MikhailetsMurach14}, which systematically expounds this interpolation (see also \cite[Section~2]{MikhailetsMurach08MFAT1}).

Let $X:=[X_{0},X_{1}]$ be an ordered pair of separable complex Hilbert spaces such that $X_{1}\subseteq X_{0}$ and this embedding is continuous and dense. This pair is said to be admissible. For $X$, there is a positive-definite self-adjoint operator $J$ on $X_{0}$ with the domain $X_{1}$ such that $\|Jv\|_{X_{0}}=\|v\|_{X_{1}}$ for every $v\in X_{1}$. This operator is uniquely determined by the pair $X$ and is called a generating operator for~$X$; see, e.g., \cite[Chapter~IV, Theorem~1.12]{KreinPetuninSemenov82}. The operator defines an isometric isomorphism $J:X_{1}\leftrightarrow X_{0}$.

Let $\mathcal{B}$ denote the set of all Borel measurable functions $\psi:(0,\infty)\rightarrow(0,\infty)$ such that $\psi$ is bounded on each compact interval $[a,b]$, with $0<a<b<\infty$, and that $1/\psi$ is bounded on every semiaxis $[a,\infty)$, with $a>0$.

Choosing a function $\psi\in\mathcal{B}$ arbitrarily, we consider the (generally, unbounded) operator $\psi(J)$ defined on $X_{0}$ as the Borel function $\psi$ of $J$. This operator is built with the help of Spectral Theorem applied to the self-adjoint operator $J$. Let $[X_{0},X_{1}]_{\psi}$ or, simply, $X_{\psi}$ denote the domain of $\psi(J)$ endowed with the inner product $(v_{1},v_{2})_{X_{\psi}}:=(\psi(J)v_{1},\psi(J)v_{2})_{X_{0}}$
and the corresponding norm $\|v\|_{X_{\psi}}:=\|\psi(J)v\|_{X_{0}}$. The linear space $X_{\psi}$ is Hilbert and separable with respect to this norm.

A function $\psi\in\mathcal{B}$ is called an interpolation parameter if the following condition is satisfied for all admissible pairs $X=[X_{0},X_{1}]$ and $Y=[Y_{0},Y_{1}]$ of Hilbert spaces and for an arbitrary linear mapping $T$ given on $X_{0}$: if the restriction of $T$ to $X_{j}$ is a bounded operator $T:X_{j}\rightarrow Y_{j}$ for each $j\in\{0,1\}$, then the restriction of $T$ to
$X_{\psi}$ is also a bounded operator $T:X_{\psi}\rightarrow Y_{\psi}$.

If $\psi$ is an interpolation parameter, then we say that the Hilbert space $X_{\psi}$ is obtained by the interpolation with the function parameter $\psi$ of the pair $X=\nobreak[X_{0},X_{1}]$ or, otherwise speaking, between the spaces $X_{0}$ and $X_{1}$. In this case, the dense and continuous embeddings $X_{1}\hookrightarrow
X_{\psi}\hookrightarrow X_{0}$ hold.

The class of all interpolation parameters (in the sense of the given definition) admits a constructive description. Namely, a function $\psi\in\mathcal{B}$ is an interpolation parameter if and only if $\psi$ is pseudoconcave in a neighbourhood of infinity. The latter property means that there exists a concave positive function $\psi_{1}(r)$ of $r\gg1$ that both the functions $\psi/\psi_{1}$ and $\psi_{1}/\psi$ are bounded in some neighbourhood of infinity. This criterion follows from Peetre's description of all interpolation functions for the weighted Lebesgue spaces \cite{Peetre66, Peetre68} (this result of Peetre is set forth in the monograph \cite[Theorem 5.4.4]{BerghLefstrem76}). The proof of the criterion is given in \cite[Section 1.1.9]{MikhailetsMurach14}.

An application of this criterion to power functions gives the classical result by Lions and S.~Krein. Namely, the function $\psi(r)\equiv r^{\theta}$ is an interpolation parameter whenever $\nobreak{0\leq\theta\leq1}$. In this case, the exponent $\theta$ serves as a number parameter of the interpolation, and the interpolation space $X_{\psi}$ is also denoted by $X_{\theta}$. This interpolation was used in formulas \eqref{8f71} and \eqref{8f72} in the special case of $\theta=1/2$.

Let us formulate some general properties of interpolation with a function parameter; they will be used in our proofs. The first of these properties enables us to reduce the interpolation of subspaces to the interpolation of the whole spaces (see \cite[Theorem~1.6]{MikhailetsMurach14} or \cite[Section~1.17.1, Theorem~1]{Triebel95}). As usual, subspaces of normed spaces are assumed to be closed. Generally, we consider nonorthogonal projectors onto subspaces of a Hilbert space.

\begin{proposition}\label{8prop1}
Let $X=[X_{0},X_{1}]$ be an admissible pair of Hilbert spaces, and let $Y_{0}$ be a subspace of $X_{0}$. Then $Y_{1}:=X_{1}\cap Y_{0}$ is a subspace of $X_{1}$. Suppose that there exists a linear mapping $P:X_{0}\rightarrow X_{0}$ such that $P$ is a projector of the space $X_{j}$ onto its subspace $Y_{j}$ for each $j\in\{0,\,1\}$. Then the pair $[Y_{0},Y_{1}]$ is admissible, and $[Y_{0},Y_{1}]_{\psi}=X_{\psi}\cap Y_{0}$ with equivalence of norms for an arbitrary interpolation parameter~$\psi\in\mathcal{B}$. Here, $X_{\psi}\cap Y_{0}$ is a subspace of $X_{\psi}$.
\end{proposition}

The second property reduces the interpolation of orthogonal sums of Hilbert spaces to the interpolation of their summands (see \cite[Theorem~1.8]{MikhailetsMurach14}.

\begin{proposition}\label{8prop2}
Let $[X_{0}^{(j)},X_{1}^{(j)}]$, with $j=1,\ldots,q$, be a finite collection of admissible pairs of Hilbert spaces. Then
$$
\biggl[\,\bigoplus_{j=1}^{q}X_{0}^{(j)},\,
\bigoplus_{j=1}^{q}X_{1}^{(j)}\biggr]_{\psi}=\,
\bigoplus_{j=1}^{q}\bigl[X_{0}^{(j)},\,X_{1}^{(j)}\bigr]_{\psi}
$$
with equality of norms for every function $\psi\in\mathcal{B}$.
\end{proposition}

The third property shows that the interpolation with a function parameter is stable with respect to its repeated fulfillment \cite[Theorem~1.3]{MikhailetsMurach14}.

\begin{proposition}\label{8prop3}
Let $\alpha,\beta,\psi\in\mathcal{B}$, and suppose that the function $\alpha/\beta$ is bounded in a neighbourhood of infinity. Define the function $\omega\in\mathcal{B}$  by the formula
$\omega(r):=\alpha(r)\psi(\beta(r)/\alpha(r))$ for $r>0$. Then $\omega\in\mathcal{B}$, and $[X_{\alpha},X_{\beta}]_{\psi}=X_{\omega}$ with equality of norms for every admissible pair $X$ of Hilbert spaces. Besides, if $\alpha,\beta,\psi$ are interpolation parameters, then
$\omega$ is also an interpolation parameter.
\end{proposition}

Our proof of Theorems \ref{8th1} and \ref{8th2} is based on the key fact that the interpolation with an appropriate function parameter between margin Sobolev spaces in \eqref{8f5a} and \eqref{8f5b} gives the intermediate H\"ormander spaces $H^{s,s/2;\varphi}(\cdot)$ and $H^{s;\varphi}(\cdot)$ respectively. Let us formulate this property separately for isotropic and for anisotropic spaces.

\begin{proposition}\label{8prop4}
Let real numbers $s_{0}$, $s$, and $s_{1}$ satisfy the inequalities
$s_{0}<s<s_{1}$, and let $\varphi\in\mathcal{M}$. Put
\begin{equation}\label{8f16}
\psi(r):=
\begin{cases}
\;r^{(s-s_{0})/(s_{1}-s_{0})}\,\varphi(r^{1/(s_{1}-s_{0})})&\text{if}
\quad r\geq1,\\
\;\varphi(1) & \text{if}\quad0<r<1.
\end{cases}
\end{equation}
Then the function $\psi\in\mathcal{B}$ is an interpolation parameter, and the equality of spaces
\begin{equation}\label{8f49}
H^{s-\lambda;\varphi}(W)=
\bigl[H^{s_{0}-\lambda}(W),H^{s_{1}-\lambda}(W)\bigr]_{\psi}
\end{equation}
holds true up to equivalence of norms for arbitrary $\lambda\in\mathbb{R}$ provided that $W=G$ or $W=\Gamma$. If $W=\mathbb{R}^{k}$ with $1\leq k\in\mathbb{Z}$, then \eqref{8f49} holds true with equality of norms in spaces.
\end{proposition}

This result is due to \cite[Theorems 3.1 and 3.5]{MikhailetsMurach06UMJ3}; see also monograph \cite[Theorems 1.14, 2.2, and 3.2]{MikhailetsMurach14} for the cases where $W=\mathbb{R}^{k}$, $W=\Gamma$, and $W=G$ respectively.

\begin{proposition}\label{8prop5}
Let real numbers $s_{0}$, $s$, and $s_{1}$ satisfy the inequalities
$0\leq s_{0}<s<s_{1}$, and let $\varphi\in\mathcal{M}$. Define an interpolation parameter $\psi\in\mathcal{B}$ by formula \eqref{8f16}. Then the equality of spaces
\begin{equation}\label{8f22}
H^{s-\lambda,(s-\lambda)/2;\varphi}(W)=
\bigl[H^{s_{0}-\lambda,(s_{0}-\lambda)/2}(W),
H^{s_{1}-\lambda,(s_{1}-\lambda)/2}(W)\bigr]_{\psi}
\end{equation}
holds true up to equivalence of norms for arbitrary real $\lambda\leq s_{0}$ provided that $W=\Omega$ or $W=S$. If $W=\mathbb{R}^{k}$ with $2\leq k\in\mathbb{Z}$, then \eqref{8f22} holds true with equality of norms in spaces without the assumption that $0\leq s_{0}$.
\end{proposition}

This result is due to \cite[Theorem~2 and Lemma~1]{Los16JMathSci} for
the cases where $W=S$ and  $W=\mathbb{R}^{k}$ respectively. In the $W=\Omega$ case, the proof of the result is the same as the proof of  its analog for a strip \cite[Lemma~2]{Los16JMathSci}.

\section{Proofs}\label{8sec6}

To deduce Theorems \ref{8th1} and \ref{8th2} from their known counterparts in the Sobolev case, we need to prove a version of Proposition \ref{8prop5} (with $\lambda=0$) for the target spaces of  isomorphisms \eqref{8f8} and \eqref{8f12}. This proof will be based on the following lemma about properties of the operator that assigns the Cauchy data to an arbitrary function $g\in H^{s,s/2;\varphi}(S)$.

\begin{lemma}\label{8lem1}
Choose an integer $r\geq1$, and consider the linear mapping
\begin{equation}\label{8f25}
R:g\mapsto\bigl(g\!\upharpoonright\!\Gamma,
\partial_{t}g\!\upharpoonright\!\Gamma,\dots,
\partial^{r-1}_{t}g\!\upharpoonright\!\Gamma\bigr),
\quad\mbox{with}\quad g\in C^{\infty}(\overline{S}).
\end{equation}
This mapping extends uniquely (by continuity) to a bounded linear operator
\begin{equation}\label{8f65}
R:H^{s,s/2;\varphi}(S)\rightarrow \bigoplus_{k=0}^{r-1}
H^{s-2k-1;\varphi}(\Gamma)=:\mathbb{H}^{s;\varphi}(\Gamma)
\end{equation}
for arbitrary $s>2r-1$ and $\varphi\in\mathcal{M}$. This operator is right invertible; moreover, there exists a linear mapping  $T:(L_2(\Gamma))^r\to L_2(S)$ that
for arbitrary $s>2r-1$ and  $\varphi\in\mathcal{M}$ the restriction of $T$ to the space $\mathbb{H}^{s;\varphi}(\Gamma)$
is a bounded linear operator
\begin{equation}\label{8f43}
T:\mathbb{H}^{s;\varphi}(\Gamma)\to H^{s,s/2;\varphi}(S)
\end{equation}
and that $RTv=v$ for every $v\in\mathbb{H}^{s;\varphi}(\Gamma)$.
\end{lemma}

\begin{proof}
We first prove an analog of this lemma for H\"ormander spaces
defined on $\mathbb{R}^{n}$ and $\mathbb{R}^{n-1}$ instead of $S$ and $\Gamma$. Then we deduce the lemma with the help of the special local charts on~$S$.

Consider the linear mapping
\begin{equation}\label{8f60}
R_{0}:w\mapsto\bigl(w\!\mid_{t=0},\,
\partial_{t}w\!\mid_{t=0},\dots,
\partial^{r-1}_{t}w\!\mid_{t=0}\bigr),\quad\mbox{with}
\quad w\in C_{0}^{\infty}(\mathbb{R}^{n}).
\end{equation}
Here, we interpret $w$ as a function $w(x,t)$ of $x\in\mathbb{R}^{n-1}$ and $t\in\mathbb{R}$ so that $R_{0}w\in(C_{0}^{\infty}(\mathbb{R}^{n-1}))^{r}$.
Choose $s>2r-1$ and $\varphi\in\mathcal{M}$ arbitrarily, and prove that  the mapping \eqref{8f60} extends uniquely (by continuity) to a bounded linear operator
\begin{equation}\label{8f63}
R_{0}:H^{s,s/2;\varphi}(\mathbb{R}^{n})\rightarrow \bigoplus_{k=0}^{r-1}
H^{s-2k-1;\varphi}(\mathbb{R}^{n-1})=:
\mathbb{H}^{s;\varphi}(\mathbb{R}^{n-1}).
\end{equation}
This fact is known in the Sobolev case of $\varphi\equiv1$ due to \cite[Chapter~II, Theorem 7]{Slobodetskii58}. Using the interpolation with a function parameter between Sobolev spaces, we can deduce this fact in the general situation of arbitrary $\varphi\in\mathcal{M}$.

Namely, choose $s_0,s_1\in\mathbb{R}$ such that $2r-1<s_0<s<s_1$ and
consider the bounded linear operators
\begin{equation}\label{8f63-sob}
R_{0}:H^{s_j,s_j/2}(\mathbb{R}^{n})\to\mathbb{H}^{s_j}(\mathbb{R}^{n-1}),
\quad\mbox{with}\quad j\in\{0,1\}.
\end{equation}
Let $\psi$ be the interpolation parameter \eqref{8f16}. Then the restriction of the mapping \eqref{8f63-sob} with $j=0$ to the space
\begin{equation}\label{8f-intH}
\bigl[H^{s_0,s_{0}/2}(\mathbb{R}^{n}),
H^{s_1,s_{1}/2}(\mathbb{R}^{n})\bigr]_{\psi}=
H^{s,s/2;\varphi}(\mathbb{R}^{n})
\end{equation}
is a bounded operator
\begin{equation}\label{8f27}
R_{0}:H^{s,s/2;\varphi}(\mathbb{R}^{n})\to
\bigl[\mathbb{H}^{s_0}(\mathbb{R}^{n-1}),
\mathbb{H}^{s_1}(\mathbb{R}^{n-1})\bigr]_{\psi}.
\end{equation}
The latter equality is due to Proposition~\ref{8prop5}. This operator is an extension by continuity of the mapping \eqref{8f60} because the set
$C_{0}^{\infty}(\mathbb{R}^{n})$ is dense in $H^{s,s/2;\varphi}(\mathbb{R}^{n})$. Owing to Propositions \ref{8prop2} and \ref{8prop4}, we get
\begin{equation}\label{8f-intHH}
\begin{aligned}
\bigl[\mathbb{H}^{s_0}(\mathbb{R}^{n-1}),
\mathbb{H}^{s_1}(\mathbb{R}^{n-1})\bigr]_{\psi}&=
\bigoplus_{k=0}^{r-1}
\bigl[H^{s_0-2k-1}(\mathbb{R}^{n-1}),
H^{s_1-2k-1}(\mathbb{R}^{n-1})\bigr]_{\psi}\\
&=\bigoplus_{k=0}^{r-1}H^{s-2k-1;\varphi}(\mathbb{R}^{n-1})=
\mathbb{H}^{s;\varphi}(\mathbb{R}^{n-1}).
\end{aligned}
\end{equation}
Hence, the linear bounded operator \eqref{8f27} is the required operator \eqref{8f63}.

Let us now build a linear mapping
\begin{equation}\label{8f58}
T_0:\bigl(L_{2}(\mathbb{R}^{n-1})\bigr)^r\to L_{2}(\mathbb{R}^{n})
\end{equation}
that its restriction to each space $\mathbb{H}^{s;\varphi}(\mathbb{R}^{n-1})$ with $s>2r-1$ and $\varphi\in\mathcal{M}$ is a bounded operator between the spaces $\mathbb{H}^{s;\varphi}(\mathbb{R}^{n-1})$ and $H^{s,s/2;\varphi}(\mathbb{R}^{n})$ and that this operator is right inverse to \eqref{8f63}.

Similarly to H\"ormander \cite[Proof of Theorem~2.5.7]{Hermander63} we define the linear mapping
\begin{equation}\label{8f58-def}
T_0:v\mapsto F_{\xi\mapsto x}^{-1}\biggl[
\beta\bigl(\langle\xi\rangle^{2}t\bigr)\,\sum_{k=0}^{r-1}
\frac{1}{k!}\,\widehat{v_k}(\xi)\times t^k\biggr](x,t)
\end{equation}
on the linear topological space of vectors
$$
v:=(v_0,\dots,v_{r-1})\in\bigl(\mathcal{S}'(\mathbb{R}^{n-1})\bigr)^r.
$$
We consider $T_{0}v$ as a distribution on the Euclidean space $\mathbb{R}^{n}$ of points $(x,t)$, with $x=(x_{1},\ldots,x_{n-1})\in\mathbb{R}^{n-1}$ and $t\in\mathbb{R}$. In \eqref{8f58-def}, the function $\beta\in C^{\infty}_{0}(\mathbb{R})$ is chosen so that $\beta=1$ in a certain neighbourhood of zero. As usual, $F_{\xi\mapsto x}^{-1}$ denotes the inverse Fourier transform with respect to $\xi=(\xi_{1},\ldots,\xi_{n-1})\in\mathbb{R}^{n-1}$, and $\langle\xi\rangle:=(1+|\xi|^2)^{1/2}$. The variable $\xi$ is dual to $x$ relative to the direct Fourier transform $\widehat{w}(\xi)=(Fw)(\xi)$ of a function $w(x)$.

Obviously, the mapping \eqref{8f58-def} is well defined and acts continuously between $(\mathcal{S}'(\mathbb{R}^{n-1})^r$ and $\mathcal{S}'(\mathbb{R}^{n})$. It is also evident that the restriction of this mapping to the space $(L_{2}(\mathbb{R}^{n-1}))^r$ is a bounded operator between $(L_{2}(\mathbb{R}^{n-1}))^r$ and $L_{2}(\mathbb{R}^{n})$.

We assert that
\begin{equation}\label{8f57}
R_{0}T_{0}v=v \quad\mbox{for every}\quad
v\in\bigl(\mathcal{S}(\mathbb{R}^{n-1})\bigr)^r.
\end{equation}
Here, as usual, $\mathcal{S}(\mathbb{R}^{n-1})$ denotes the linear topological space of all rapidly decreasing infinitely smooth functions on $\mathbb{R}^{n-1}$. Since $v\in(\mathcal{S}(\mathbb{R}^{n-1})^r$ implies $T_{0}v\in\mathcal{S}(\mathbb{R}^{n-1})$, the left-hand side of the equality \eqref{8f57} is well defined. Let us prove this equality.

Choosing $j\in\{0,\dots,r-1\}$ and
$v=(v_0,\dots,v_{r-1})\in(S(\mathbb{R}^{n-1}))^r$
arbitrarily, we get
\begin{align*}
F\bigl[\partial^j_tT_0v\!\mid_{t=0}\bigr](\xi)&=
\partial^j_{t}F_{x\mapsto\xi}[T_0v](\xi,t)\big|_{t=0}=
\partial^j_t\biggl(
\beta\bigl(\langle\xi\rangle^{2}t\bigr)\,\sum_{k=0}^{r-1}
\frac{1}{k!}\,\widehat{v_k}(\xi)\,t^k\biggr)\bigg|_{t=0}\\
&=\beta(0)\biggl(\partial^j_t\sum_{k=0}^{r-1}
\frac{1}{k!}\,\widehat{v_k}(\xi)\,t^k\biggr)\!
\bigg|_{t=0}=\beta(0)\,j!\,\frac{1}{j!}\,\widehat{v_j}(\xi)=
\widehat{v_j}(\xi)
\end{align*}
for every $\xi\in\mathbb{R}^{n-1}$. In the fourth equality, we have used the fact that $\beta=1$ in a neighbourhood of zero.
Thus, the Fourier transforms of all components of the vectors $R_{0}T_{0}v$ and $v$ coincide, which is equivalent to \eqref{8f57}.

Let us now prove that the restriction of the mapping \eqref{8f58-def} to each space
\begin{equation}\label{8f-doubleSobolev}
\mathbb{H}^{2m}(\mathbb{R}^{n-1})=
\bigoplus_{k=0}^{r-1}H^{2m-2k-1}(\mathbb{R}^{n-1})
\end{equation}
with $0\leq m\in\mathbb{Z}$ is a bounded operator between $\mathbb{H}^{2m}(\mathbb{R}^{n-1})$ and $H^{2m,m}(\mathbb{R}^{n})$.
Note that the integers $2m-2k-1$ may be negative in \eqref{8f-doubleSobolev}.

Let an integer $m\geq0$. We make use of the fact that the norm in the space $H^{2m,m}(\mathbb{R}^{n})$ is equivalent to the norm
\begin{equation*}
\|w\|_{2m,m}:=\biggl(\|w\|^2+
\sum_{j=1}^{n-1}\|\partial_{x_j}^{2m}w\|^2+
\|\partial_{t}^{m}w\|^2\biggr)^{1/2}
\end{equation*}
(see, e.g., \cite[Section~9.1]{BesovIlinNikolskii75}). Here and below in this proof, $\|\cdot\|$ stands for the norm in the Hilbert space $L_2(\mathbb{R}^{n})$. Of course, $\partial_{x_j}u$ and $\partial_{t}$ denote the operators of generalized partial derivatives with respect to $x_j$ and $t$ respectively. Choosing $v=(v_0,\dots,v_{r-1})\in(\mathcal{S}(\mathbb{R}^{n-1}))^r$ arbitrarily and using the Parseval equality, we obtain the following:
\begin{align*}
\|T_0v\|^2_{2m,m}&=\|T_0v\|^2+
\sum_{j=1}^{n-1}\|\partial_{x_j}^{2m}\,T_0v\|^2+
\|\partial_{t}^{m}\,T_0v\|^2\\
&=\|\widehat{T_0v}\|^2+
\sum_{j=1}^{n-1}\|\xi_j^{2m}\,\widehat{T_0v}\|^2+
\|\partial_{t}^{m}\,\widehat{T_0v}\|^2\\
&\leq\sum_{k=0}^{r-1}\frac{1}{k!}\,
\int\limits_{\mathbb{R}^{n}}
\bigl|\beta(\langle\xi\rangle^{2}t)\,\widehat{v_k}(\xi)\,t^k\bigr|^2
d\xi dt\\
&+\sum_{j=1}^{n-1}\sum_{k=0}^{r-1}\frac{1}{k!}\,
\int\limits_{\mathbb{R}^{n}}
\bigl|\xi_j^{2m}\,\beta(\langle\xi\rangle^{2}t)\,\widehat{v_k}(\xi)\,
t^k\bigr|^2d\xi dt\\
&+\sum_{k=0}^{r-1}\frac{1}{k!}\,
\int\limits_{\mathbb{R}^{n}}
\bigl|\partial_{t}^{m}\bigl(\beta(\langle\xi\rangle^{2}t)\,t^k\bigr)\,
\widehat{v_k}(\xi)\bigr|^2d\xi dt.
\end{align*}

Let us estimate each of these three integrals separately. We begin with the third integral. Changing the variable $\tau=\langle\xi\rangle^{2}t$ in the interior integral with respect to $t$, we get the equalities
\begin{align*}
\int\limits_{\mathbb{R}^{n}}
\bigl|\partial_{t}^{m}\bigl(\beta(\langle\xi\rangle^{2}t)\,t^k\bigr)\,
\widehat{v_k}(\xi)\bigr|^2d\xi dt
&=\int\limits_{\mathbb{R}^{n-1}}
|\widehat{v_k}(\xi)|^2d\xi \int\limits_{\mathbb{R}}
|\partial_{t}^{m}(\beta(\langle\xi\rangle^{2}t)t^k)|^2 dt\\
&=\int\limits_{\mathbb{R}^{n-1}}
\langle\xi\rangle^{4m-4k-2}\,|\widehat{v_k}(\xi)|^2 d\xi \int\limits_{\mathbb{R}}
|\partial_{\tau}^{m}(\beta(\tau)\tau^{k})|^2 d\tau.
\end{align*}
Hence,
$$
\int\limits_{\mathbb{R}^{n}}
\bigl|\partial_{t}^{m}\bigl(\beta(\langle\xi\rangle^{2}t)\,t^k\bigr)\,
\widehat{v_k}(\xi)\bigr|^2d\xi dt=
c_1\,\|v_k\|^2_{H^{2m-2k-1}(\mathbb{R}^{n-1})},
$$
with
$$
c_1:=\int\limits_{\mathbb{R}}
|\partial_{\tau}^{m}(\beta(\tau)\tau^{k})|^2 d\tau<\infty.
$$

Using the same changing of the variable $t$ in the second integral, we obtain the following:
\begin{align*}
\int\limits_{\mathbb{R}^{n}}
\bigl|\xi_j^{2m}\,\beta(\langle\xi\rangle^{2}t)\,\widehat{v_k}(\xi)\,
t^k\bigr|^2d\xi dt
&=\int\limits_{\mathbb{R}^{n-1}}
|\xi_j|^{4m}|\widehat{v_k}(\xi)|^2d\xi \int\limits_{\mathbb{R}}
|t^{k}\,\beta(\langle\xi\rangle^{2}t)|^2dt\\
&=\int\limits_{\mathbb{R}^{n-1}}
|\xi_j|^{4m}\langle\xi\rangle^{-4k-2}\,|\widehat{v_k}(\xi)|^2d\xi \int\limits_{\mathbb{R}}|\tau^{k}\beta(\tau)|^2d\tau\\
&\leq\int\limits_{\mathbb{R}^{n-1}}
\langle\xi\rangle^{4m-4k-2}\,|\widehat{v_k}(\xi)|^2d\xi
\int\limits_{\mathbb{R}}|\tau^{k}\beta(\tau)|^2d\tau.
\end{align*}
Hence,
$$
\int\limits_{\mathbb{R}^{n}}
\bigl|\xi_j^{2m}\,\beta(\langle\xi\rangle^{2}t)\,\widehat{v_k}(\xi)\,
t^k\bigr|^2d\xi dt
\leq c_2\,\|v_k\|^2_{H^{2m-2k-1}(\mathbb{R}^{n-1})},
$$
with
$$
c_2:=\int\limits_{\mathbb{R}}|\tau^{k}\beta(\tau)|^2d\tau<\infty.
$$

Finally, replacing the symbol $\xi_j$ with $1$ in the previous reasoning,
we obtain the following estimate for the first integral:
$$
\int\limits_{\mathbb{R}^{n}}
\bigl|\beta(\langle\xi\rangle^{2}t)\,\widehat{v_k}(\xi)\,
t^k\bigr|^2d\xi dt
\leq c_2\,\|v_k\|^2_{H^{-2k-1}(\mathbb{R}^{n-1})}
\leq c_2\,\|v_k\|^2_{H^{2m-2k-1}(\mathbb{R}^{n-1})}.
$$

Thus, we conclude that
\begin{equation*}
\|T_0v\|_{H^{2m,m}(\mathbb{R}^{n})}^{2}\leq c\,\sum_{k=0}^{r-1}
\|v_k\|^2_{H^{2m-2k-1}(\mathbb{R}^{n-1})}=
c\,\|v\|_{\mathbb{H}^{2m}(\mathbb{R}^{n-1})}^{2}
\end{equation*}
for any $v\in(\mathcal{S}(\mathbb{R}^{n-1}))^r$, with the number $c>0$ being independent of $v$. Since the set $\bigl(S(\mathbb{R}^{n-1})\bigr)^r$ is dense in $\mathbb{H}^{2m}(\mathbb{R}^{n-1})$, it follows from the latter estimate that the mapping \eqref{8f58-def} sets a bounded linear operator
\begin{equation*}
T_0:\mathbb{H}^{2m}(\mathbb{R}^{n-1})\to H^{2m,m}(\mathbb{R}^{n})
\quad\mbox{whenever}\quad 0\leq m\in\mathbb{Z}.
\end{equation*}

Let us deduce from this fact that the mapping \eqref{8f58-def} acts continuously between the spaces $\mathbb{H}^{s;\varphi}(\mathbb{R}^{n-1})$ and $H^{s,s/2;\varphi}(\mathbb{R}^{n})$ for every $s>2r-1$ and $\varphi\in\mathcal{M}$. Put $s_0=0$, choose an even integer $s_1>s$, and consider the linear bounded operators
\begin{equation}\label{8f66}
T_{0}:\mathbb{H}^{s_j}(\mathbb{R}^{n-1})\to H^{s_j,s_j/2}(\mathbb{R}^{n}),
\quad\mbox{with}\quad j\in\{0,1\}.
\end{equation}
Let, as above, $\psi$ be the interpolation parameter \eqref{8f16}. Then the restriction of the mapping \eqref{8f66} with $j=0$ to the space
\begin{equation*}
\bigl[\mathbb{H}^{s_0}(\mathbb{R}^{n-1}),
\mathbb{H}^{s_1}(\mathbb{R}^{n-1})\bigr]_{\psi}=
\mathbb{H}^{s;\varphi}(\mathbb{R}^{n-1})
\end{equation*}
is a bounded operator
\begin{equation}\label{8f48}
T_0:\mathbb{H}^{s;\varphi}(\mathbb{R}^{n-1})\to
H^{s,s/2;\varphi}(\mathbb{R}^{n}).
\end{equation}
Here, we have used formulas \eqref{8f-intH} and \eqref{8f-intHH}, which remain true for the considered $s_{0}$ and $s_{1}$.

Now the equality \eqref{8f57} extends by continuity over all vectors
$v\in\mathbb{H}^{s;\varphi}(\mathbb{R}^{n-1})$. Hence, the operator \eqref{8f48} is right inverse to \eqref{8f63}. Thus, the required mapping \eqref{8f58} is built.

We need to introduce analogs of the operators \eqref{8f63} and \eqref{8f48} for the strip
$$
\Pi=\bigl\{(x,t):x\in\mathbb{R}^{n-1},0<t<\tau\bigr\}.
$$
Let $s>2r-1$ and $\varphi\in\mathcal{M}$. Given $u\in H^{s,s/2;\varphi}(\Pi)$, we put
$R_{1}u:=R_{0}w$, where a function
$w\in H^{s,s/2;\varphi}(\mathbb{R}^{n})$ satisfies the condition
$w\!\upharpoonright\!\Pi=u$. Evidently, this definition does not depend on the choice of $w$. The linear mapping $u\mapsto R_{1}u$ is a bounded  operator
\begin{equation}\label{8f67}
R_{1}:H^{s,s/2;\varphi}(\Pi)\to\mathbb{H}^{s;\varphi}(\mathbb{R}^{n-1}).
\end{equation}
This follows immediately from the boundedness of the operator \eqref{8f63} and from the definition of the norm in  $H^{s,s/2;\varphi}(\Pi)$.

Let us introduce a right-inverse of \eqref{8f67} on the base of the mapping \eqref{8f58-def}. We put $T_{1}v:=(T_0v)\!\upharpoonright\!\Pi$ for arbitrary $v\in(L_{2}(\mathbb{R}^{n-1}))^{r}$. The restriction of the linear mapping $v\mapsto T_{1}v$ over vectors $v\in\mathbb{H}^{s;\varphi}(\mathbb{R}^{n-1})$ is a bounded operator
\begin{equation}\label{8f51}
T_1:\mathbb{H}^{s;\varphi}(\mathbb{R}^{n-1})\to H^{s,s/2;\varphi}(\Pi).
\end{equation}
This follows directly from the boundedness of the operator \eqref{8f48}. Observe that
$$
R_1T_1v=R_1\bigl((T_0v)\!\upharpoonright\!\Pi\bigr)=R_0T_0v=v
\quad\mbox{for every}\quad v\in\mathbb{H}^{s;\varphi}(\mathbb{R}^{n-1}).
$$
Thus, the operator \eqref{8f51} is right inverse to \eqref{8f67}.

Using operators \eqref{8f67} and \eqref{8f51}, we can now prove our lemma with the help of the special local charts \eqref{8f-local} on $S$. As above, let $s>2r-1$ and $\varphi\in\mathcal{M}$. Choosing $k\in\{0,\dots,r-1\}$ and $g\in C^{\infty}(\overline{S})$ arbitrarily, we get the following:
\begin{align*}
\|\partial^{k}_{t}g\!\upharpoonright\!\Gamma\|_
{H^{s-2k-1;\varphi}(\Gamma)}^{2}&=
\sum_{j=1}^{\lambda}
\|\bigl(\chi_{j}(\partial^{k}_{t}g\!\upharpoonright\!\Gamma)\bigr)
\circ\theta_{j}\|_{H^{s-2k-1;\varphi}(\mathbb{R}^{n-1})}^{2}\\
&=\sum_{j=1}^{\lambda}
\|\partial^{k}_{t}\bigl((\chi_{j}\,g)\circ\theta^{\ast}_{j}\bigr)
\!\upharpoonright\!\mathbb{R}^{n-1})\|_
{H^{s-2k-1;\varphi}(\mathbb{R}^{n-1})}^{2}\\
&\leq c^{2}\,\sum_{j=1}^{\lambda}
\|(\chi_{j}\,g)\circ\theta^{\ast}_{j}\|_{H^{s,s/2;\varphi}(\Pi)}^{2}=
c^{2}\,\|g\|_{H^{s,s/2;\varphi}(S)}^{2}.
\end{align*}
Here, $c$ denotes the norm of the bounded operator \eqref{8f67}, and, as usual, symbol "$\circ$" designates a composition of functions. Recall that $\{\theta_{j}\}$ is a collection of local charts on $\Gamma$ and that $\{\chi_{j}\}$ is an infinitely smooth partition of unity on $\Gamma$. Thus,
\begin{equation*}
\|Rg\|_{\mathbb{H}^{s;\varphi}(\Gamma)}\leq c\,\sqrt{r}\,\|g\|_{H^{s,s/2;\varphi}(S)}
\quad\mbox{for every}\quad g\in C^{\infty}(\overline{S}).
\end{equation*}
This implies that the mapping \eqref{8f25} extends by continuity to the bounded linear operator \eqref{8f65}.

Let us build the linear mapping $T:(L_2(\Gamma))^r\to L_2(S)$ whose restriction to $\mathbb{H}^{s;\varphi}(\Gamma)$ is a right-inverse of \eqref{8f65}. Consider the linear mapping of flattening of $\Gamma$
\begin{equation*}
L:v\mapsto\bigl((\chi_{1}v)\circ\theta_{1},\ldots,
(\chi_{\lambda}v)\circ\theta_{\lambda}\bigr),
\quad\mbox{with}\quad v\in L_2(\Gamma).
\end{equation*}
Its restriction to $H^{\sigma;\varphi}(\Gamma)$ is an isometric operator
\begin{equation}\label{8f52}
L:H^{\sigma;\varphi}(\Gamma)\rightarrow
\bigl(H^{\sigma;\varphi}(\mathbb{R}^{n-1})\bigr)^{\lambda}
\quad\mbox{whenever}\quad\sigma>0.
\end{equation}
Besides, consider the linear mapping of sewing of $\Gamma$
\begin{equation*}
K:(h_{1},\ldots,h_{\lambda})\mapsto\sum_{j=1}^{\lambda}\,
O_{j}\bigl((\eta_{j}h_{j})\circ\theta_{j}^{-1}\bigr),
\quad\mbox{with}\quad h_{1},\ldots,h_{\lambda}\in L_2(\mathbb{R}^{n-1}).
\end{equation*}
Here, each function $\eta_{j}\in
C_{0}^{\infty}(\mathbb{R}^{n-1})$ is chosen so that $\eta_{j}=1$ on the
set $\theta^{-1}_{j}(\mathrm{supp}\,\chi_{j})$, whereas $O_{j}$
denotes the operator of the extension by zero to $\Gamma$ of a function given on $\Gamma_j$. The restriction of this mapping to $(H^{\sigma;\varphi}(\mathbb{R}^{n-1}))^{\lambda}$ is a bounded  operator
\begin{equation*}
K:\bigl(H^{\sigma;\varphi}(\mathbb{R}^{n-1})\bigr)^{\lambda}\to
H^{\sigma;\varphi}(\Gamma)\quad\mbox{whenever}\quad\sigma>0,
\end{equation*}
and this operator is left inverse to \eqref{8f52} (see \cite[the proof of Theorem~2.2]{MikhailetsMurach14}).

The mapping $K$ induces the operator $K_{1}$ of the sewing of the manifold $S=\Gamma\times(0,\tau)$ by the formula
\begin{equation*}
\bigl(K_1(g_1,\dots,g_\lambda)\bigr)(x,t):=
\bigl(K(g_1(\cdot,t),\ldots,g_\lambda(\cdot,t))\bigr)(x)
\end{equation*}
for arbitrary functions $g_1,\dots,g_\lambda\in L_2(\Pi)$ and almost all $x\in\Gamma$ and $t\in(0,\tau)$. The restriction of the mapping $K_{1}$ to $(H^{\sigma,\sigma/2;\varphi}(\Pi))^{\lambda}$ is a bounded  operator
\begin{equation}\label{8f53}
K_{1}:(H^{\sigma,\sigma/2;\varphi}(\Pi))^{\lambda}\to
H^{\sigma,\sigma/2;\varphi}(S)\quad\mbox{whenever}\quad\sigma>0
\end{equation}
(see \cite[the proof of Theorem~2]{Los16JMathSci}).

Given $v:=(v_0,v_1,\dots,v_{r-1})\in(L_{2}(\Gamma))^{r}$, we set
\begin{equation*}
Tv:=K_1\bigl(T_1(v_{0,1},\ldots,v_{r-1,1}),\ldots,
T_1(v_{0,\lambda},\ldots,v_{r-1,\lambda})\bigr),
\end{equation*}
where
$$
(v_{k,1},\ldots,v_{k,\lambda}):=
Lv_{k}\in(L_{2}(\mathbb{R}^{n-1}))^{\lambda}
$$
for each integer $k\in\{0,\ldots,r-1\}$. The linear mapping $v\mapsto Tv$ acts continuously between $(L_{2}(\Gamma))^{r}$ and $L_{2}(S)$, which follows directly from the definitions of $L$, $T_1$, and~$K_1$.
The restriction of this mapping to $\mathbb{H}^{s;\varphi}(\Gamma)$ is the bounded operator \eqref{8f43}. This follows immediately from the boundedness of the operators \eqref{8f51}, \eqref{8f52}, and \eqref{8f53}. The operator \eqref{8f43} is right inverse to \eqref{8f65}. Indeed, choosing a vector $v=(v_0,v_1,\dots,v_{r-1})\in\mathbb{H}^{s;\varphi}(\Gamma)$ arbitrarily, we obtain the following equalities:
\begin{align*}
(RTv)_k&=\bigl(RK_1\bigl(T_1(v_{0,1},\ldots,v_{r-1,1}),\ldots,
T_1(v_{0,\lambda},\ldots,v_{r-1,\lambda})\bigr)\bigr)_k\\
&=K\bigl(\bigl(R_1T_1(v_{0,1},\ldots,v_{r-1,1})\bigr)_k,\ldots,
\bigl(R_1T_1(v_{0,\lambda},\ldots,v_{r-1,\lambda})\bigr)_k\bigr)\\
&=K(v_{k,1},\dots,v_{k,\lambda})=KLv_k=v_k.
\end{align*}
Here, the index $k$ runs over the set $\{0,\dots,r-1\}$ and denotes the $k$-th component of a vector. Hence, $RTv=v$.
\end{proof}

Using this lemma, we will now prove a version of Proposition \ref{8prop5} for the target spaces of  isomorphisms \eqref{8f8} and \eqref{8f12}. Note that the number of the compatibility conditions \eqref{8f10} and \eqref{8f14} are constant respectively on the intervals
$$
J_{0,1}:=(2,\,7/2),\quad J_{0,r}:=(2r-1/2,\,2r+3/2),\;\;\mbox{with}\;\; 2\leq r\in\mathbb{Z},
$$
and
$$
J_{1,0}:=(2,5/2),\quad J_{1,r}:=(2r+1/2,\,2r+5/2),\;\;\mbox{with}\;\; 1\leq r\in\mathbb{Z},
$$
of the varying of $s$. Namely, if $s$ ranges over some $J_{l,r}$, then this number equals $r$.

\begin{lemma}\label{8lem2}
Let $l\in\{0,\,1\}$ and $1\leq r\in\mathbb{Z}$. Suppose that real numbers $s_0,s,s_1\in J_{l,r}$ satisfy the inequality $s_0<s<s_1$ and that $\varphi\in\mathcal{M}$. Define an interpolation parameter $\psi\in\mathcal{B}$ by formula \eqref{8f16}. Then the equality of spaces
\begin{equation}\label{8f28}
\mathcal{Q}_{l}^{s-2,s/2-1;\varphi}=
\bigl[\mathcal{Q}_{l}^{s_0-2,s_{0}/2-1},
\mathcal{Q}_{l}^{s_1-2,s_{1}/2-1}\bigr]_{\psi}
\end{equation}
holds true up to equivalence of norms.
\end{lemma}

\begin{proof}
Recall that $\mathcal{Q}_{l}^{s-2,s/2-1;\varphi}$ and $\mathcal{Q}_{l}^{s_j-2,s_{j}/2-1}$, with $j\in\{0,\,1\}$, are subspaces of the Hilbert spaces $\mathcal{H}_{l}^{s-2,s/2-1;\varphi}$ and $\mathcal{H}_{l}^{s_j-2,s_{j}/2-1}$ respectively. According to Propositions \ref{8prop2}, \ref{8prop4}, and \ref{8prop5}
we obtain the following:
\begin{align*}
\bigl[&\mathcal{H}_l^{s_0-2,s_0/2-1},
\mathcal{H}_l^{s_1-2,s_1/2-1}\bigr]_{\psi} \\
&=\bigl[H^{s_0-2,s_0/2-1}(\Omega)\oplus
H^{s_0-(2l+1)/2,s_0/2-(2l+1)/4}(S)
\oplus H^{s_0-1}(G), \\
&\quad\;\;\,H^{s_1-2,s_1/2-1}(\Omega)\oplus
H^{s_1-(2l+1)/2,s_1/2-(2l+1)/4}(S)
\oplus H^{s_1-1}(G)\bigr]_{\psi}\\
&=\bigl[H^{s_0-2,s_0/2-1}(\Omega),
H^{s_1-2,s_1/2-1}(\Omega)\bigr]_{\psi}\\
&\qquad\oplus
\bigl[H^{s_0-(2l+1)/2,s_0/2-(2l+1)/4}(S),
H^{s_1-(2l+1)/2,s_1/2-(2l+1)/4}(S)\bigr]_{\psi}\\
&\qquad
\oplus\bigl[H^{s_0-1}(G),H^{s_1-1}(G)\bigr]_{\psi}\\
&=
H^{s-2,s/2-1;\varphi}(\Omega)\oplus
H^{s-(2l+1)/2,s/2-(2l+1)/4;\varphi}(S)
\oplus H^{s-1;\varphi}(G)
=\mathcal{H}_l^{s-2,s/2-1;\varphi}.
\end{align*}
Thus,
\begin{equation}\label{8f30}
\bigl[\mathcal{H}_l^{s_0-2,s_0/2-1},
\mathcal{H}_l^{s_1-2,s_1/2-1}\bigr]_{\psi}=
\mathcal{H}_l^{s-2,s/2-1;\varphi}
\end{equation}
up to equivalence of norms.

We will deduce the required formula \eqref{8f28} from \eqref{8f30} with the help of Proposition~\ref{8prop1}. To this end, we need to present a linear mapping $P$ on $\mathcal{H}_l^{s_0-2,s_0/2-1}$ such that $P$ is a projector of the space $\mathcal{H}_l^{s_j-2,s_j/2-1}$ onto its subspace $\mathcal{Q}_l^{s_j-2,s_j/2-1}$ for each $j\in\{0,\,1\}$. If we have this mapping, we will get
\begin{align*}
\bigl[\mathcal{Q}_{l}^{s_0-2,s_{0}/2-1},
\mathcal{Q}_{l}^{s_1-2,s_{1}/2-1}\bigr]_{\psi}
&=\bigl[\mathcal{H}_l^{s_0-2,s_0/2-1},
\mathcal{H}_l^{s_1-2,s_1/2-1}\bigr]_{\psi}\cap
\mathcal{Q}_{l}^{s_0-2,s_{0}/2-1}\\
&=\mathcal{H}_l^{s-2,s/2-1;\varphi}\cap
\mathcal{Q}_{l}^{s_0-2,s_{0}/2-1}\\
&=\mathcal{Q}_{l}^{s-2,s/2-1;\varphi}
\end{align*}
due to Proposition~\ref{8prop1}, formula \eqref{8f30}, and the conditions $s_0,s\in J_{l,r}$ and $s_0<s$. Note that these conditions imply the last equality because the elements of the spaces $\mathcal{Q}_{l}^{s_0-2,s_{0}/2-1}$
and $\mathcal{Q}_{l}^{s-2,s/2-1;\varphi}$ satisfy the same
compatibility conditions and because $\mathcal{H}_l^{s-2,s/2-1;\varphi}$ is embedded continuously in $\mathcal{H}_l^{s_{0}-2,s_{0}/2-1}$.

We will build the above-mentioned mapping $P$ with the help of Lemma~\ref{8lem1}. Consider first the case of $l=0$. Given $(f,g,h)\in\mathcal{H}_0^{s_{0}-2,s_{0}/2-1}$, we put
\begin{equation*}
g^*:=g+T\bigl(v_{0}\!\upharpoonright\!\Gamma-g\!\upharpoonright\!\Gamma,
\dots,v_{r-1}\!\upharpoonright\!\Gamma-\partial^{\,r-1}_t g\!\upharpoonright\!\Gamma\bigr).
\end{equation*}
Here, the functions $v_k\in H^{s_{0}-1-2k}(G)$, with $k=0,\ldots,r-1$, are defined by the recurrent formula \eqref{8f69}, and the mapping $T$ is taken from Lemma~\ref{8lem1}. The linear mapping
$P:(f,g,h)\mapsto (f,g^*,h)$ defined on all vectors $(f,g,h)\in\mathcal{H}_0^{s_{0}-2,s_{0}/2-1}$ is required. Indeed, its
restriction to each space $\mathcal{H}_0^{s_j-2,s_j/2-1}$, with $j\in\{0,\,1\}$, is a bounded operator on this space. This follows directly from Lemma~\ref{8lem1} in which we take $s:=s_j-1/2$.
Moreover, if $(f,g,h)\in\mathcal{Q}_0^{s_j-2,s_j/2-1}$, then
$P(f,g,h)=(f,g,h)$ due to the compatibility conditions \eqref{8f10}.

Consider now the case of $l=1$. Given $(f,g,h)\in\mathcal{H}_1^{s_{0}-2,s_{0}/2-1}$, we put
\begin{equation*}
g^*:=g+T\bigl(B_0[v_{0}]\!\upharpoonright\!\Gamma-g\!\upharpoonright\!\Gamma,\dots,
B_{r-1}[v_{0},\dots,v_{r-1}]\!\upharpoonright\!\Gamma-\partial^{r-1}_t g\!\upharpoonright\!\Gamma\bigr).
\end{equation*}
Here, the functions $v_0,\ldots,v_{r-1}$ and mapping $T$ are the same as in the $l=0$ case. The linear mapping
$P:(f,g,h)\mapsto (f,g^*,h)$ defined on all vectors $(f,g,h)\in\mathcal{H}_1^{s_{0}-2,s_{0}/2-1}$ is required. Indeed, its restriction to each space $\mathcal{H}_1^{s_j-2,s_j/2-1}$, with $j\in\{0,\,1\}$, is a bounded operator on this space due to Lemma~\ref{8lem1} in which $s:=s_j-3/2$. Moreover, if $(f,g,h)\in\mathcal{Q}_1^{s_j-2,s_j/2-1}$, then $P(f,g,h)=(f,g,h)$ by the compatibility conditions \eqref{8f14}.
\end{proof}

\begin{remark}\label{8rem1}
If $l=1$ and $r=0$, then the conclusion of Lemma \ref{8lem2} remains true. Indeed, in this case $\mathcal{Q}_{1}^{s-2,s/2-1;\varphi}=\mathcal{H}_{1}^{s-2,s/2-1;\varphi}$ and $\mathcal{Q}_{1}^{s_j-2,s_{j}/2-1}=\mathcal{H}_{1}^{s_j-2,s_{j}/2-1}$ for each $j\in\{0,\,1\}$ so that
\eqref{8f28} coincides with the equality \eqref{8f30}. The latter is valid in the case considered as well.
\end{remark}

Now we are in position to prove the main results of the paper.

\begin{proof}[Proofs of Theorems \ref{8th1} and \ref{8th2}]
Let $s>2$, $\varphi\in\mathcal{M}$, and $l\in\{0,\,1\}$. If $l=0$ [or $l=1$], then our reasoning relates to Theorem \ref{8th1} [or Theorem \ref{8th2}]. We first consider the case where $s\notin E_{l}$. Then $s\in J_{l,r}$ for a certain integer $r$. Choose numbers $s_0,s_1\in J_{l,r}$ such that $s_0<s<s_1$. According to Lions and Magenes \cite[Theorem~6.2]{LionsMagenes72ii}, the mapping
\begin{equation}\label{8fmap-smooth}
u\mapsto\Lambda_{l}u,\quad\mbox{with}\quad u\in C^{\infty}(\overline{\Omega}),
\end{equation}
extends uniquely (by continuity) to an isomorphism
\begin{equation}\label{8f35}
\Lambda_l:H^{s_j,s_j/2}(\Omega)\leftrightarrow
\mathcal{Q}_{l}^{s_j-2,s_j/2-1}
\quad\mbox{for each}\quad j\in\{0,1\}.
\end{equation}
Let $\psi$ be the interpolation parameter from Proposition~\ref{8prop4}. Then the restriction of the operator
\eqref{8f35} with $j=0$ to the space
\begin{equation*}
\bigl[H^{s_0,s_{0}/2}(\Omega),
H^{s_1,s_{1}/2}(\Omega)\bigr]_{\psi}=
H^{s,s/2;\varphi}(\Omega)
\end{equation*}
is an isomorphism
\begin{equation}\label{8f36}
\Lambda_l:H^{s,s/2;\varphi}(\Omega)\leftrightarrow
\bigl[\mathcal{Q}_l^{s_0-2,s_0/2-1},
\mathcal{Q}_l^{s_1-2,s_1/2-1}\bigr]_{\psi}=\mathcal{Q}_{l}^{s-2,s/2-1;\varphi}.
\end{equation}
Here, the equalities of spaces hold true up to equivalence of norms due to Proposition~\ref{8prop5} and Lemma~\ref{8lem2} (see also Remark~\ref{8rem1}). The operator \eqref{8f36} is an extension by continuity
of the mapping \eqref{8fmap-smooth} because $C^{\infty}(\overline{\Omega})$
is dense in $H^{s,s/2;\varphi}(\Omega)$. Thus, Theorems \ref{8th1} and \ref{8th2} are proved in the case considered.

Consider now the case where $s\in E_{l}$. Choose $\varepsilon\in(0,1/2)$ arbitrarily. Since $s\pm\varepsilon\notin E_{l}$ and $s-\varepsilon>2$, we have the isomorphisms
\begin{equation*}
\Lambda_l:H^{s\pm\varepsilon,(s\pm\varepsilon)/2;\varphi}(\Omega)\leftrightarrow
\mathcal{Q}_l^{s\pm\varepsilon-2,(s\pm\varepsilon)/2-1;\varphi}.
\end{equation*}
They imply that the mapping \eqref{8fmap-smooth} extends uniquely (by continuity) to an isomorphism
\begin{equation*}
\begin{aligned}
\Lambda_l:&
\bigl[H^{s-\varepsilon,(s-\varepsilon)/2;\varphi}(\Omega),
H^{s+\varepsilon,(s+\varepsilon)/2;\varphi}(\Omega)\bigr]_{1/2}\\
&\leftrightarrow
\bigl[\mathcal{Q}_l^{s-\varepsilon-2,(s-\varepsilon)/2-1;\varphi},
\mathcal{Q}_l^{s+\varepsilon-2,(s+\varepsilon)/2-1;\varphi}\bigr]_{1/2}=
\mathcal{Q}_l^{s-2,s/2-1;\varphi}.
\end{aligned}
\end{equation*}
Recall that the last equality is the definition of the space $\mathcal{Q}_l^{s-2,s/2-1;\varphi}$.

It remains to prove that
\begin{equation}\label{8f73}
H^{s,s/2;\varphi}(\Omega)=
\bigl[H^{s-\varepsilon,(s-\varepsilon)/2;\varphi}(\Omega),
H^{s+\varepsilon,(s+\varepsilon)/2;\varphi}(\Omega)\bigr]_{1/2}
\end{equation}
up to equivalence of norms. We reduce the interpolation of H\"ormander spaces to an interpolation of Sobolev spaces with the help of Proposition~\ref{8prop3}. Let us choose real $\delta>0$ such that $s-\varepsilon-\delta>0$.
According to Proposition~\ref{8prop5} we have the equalities
\begin{equation*}
H^{s-\varepsilon,(s-\varepsilon)/2;\varphi}(\Omega)=
\bigl[H^{s-\varepsilon-\delta,(s-\varepsilon-\delta)/2}(\Omega),
H^{s+\varepsilon+\delta,(s+\varepsilon+\delta)/2}(\Omega)\bigr]_{\alpha}
\end{equation*}
and
\begin{equation*}
H^{s+\varepsilon,(s+\varepsilon)/2;\varphi}(\Omega)=
\bigl[H^{s-\varepsilon-\delta,(s-\varepsilon-\delta)/2}(\Omega),
H^{s+\varepsilon+\delta,(s+\varepsilon+\delta)/2}(\Omega)\bigr]_{\beta}.
\end{equation*}
Here, the interpolation parameters $\alpha$ and $\beta$
are defined by the formulas
\begin{equation*}
\alpha(r):=r^{\delta/(2\varepsilon+2\delta)}\varphi(r^{1/(2\varepsilon+2\delta)}),
\quad
\beta(r):=r^{(2\varepsilon+\delta)/(2\varepsilon+2\delta)}\varphi(r^{1/(2\varepsilon+2\delta)})
\quad\mbox{if}\quad r\geq1
\end{equation*}
and $\alpha(r)=\beta(r):=1$ if $0<r<1$. Therefore, owing to Propositions \ref{8prop3} and \ref{8prop5}, we get
\begin{align*}
\bigl[&H^{s-\varepsilon,(s-\varepsilon)/2;\varphi}(\Omega),
H^{s+\varepsilon,(s+\varepsilon)/2;\varphi}(\Omega)\bigr]_{1/2}\\
\notag
&=\Bigl[
\bigl[H^{s-\varepsilon-\delta,(s-\varepsilon-\delta)/2}(\Omega),
H^{s+\varepsilon+\delta,(s+\varepsilon+\delta)/2}(\Omega)\bigr]_{\alpha},\\
&\quad\quad\bigl[H^{s-\varepsilon-\delta,(s-\varepsilon-\delta)/2}(\Omega),
H^{s+\varepsilon+\delta,(s+\varepsilon+\delta)/2}(\Omega)\bigr]_{\beta}
\Bigr]_{1/2}\\ \label{8f74}
&=\bigl[H^{s-\varepsilon-\delta,(s-\varepsilon-\delta)/2}(\Omega),
H^{s+\varepsilon+\delta,(s+\varepsilon+\delta)/2}(\Omega)\bigr]_{\omega}=H^{s,s/2;\varphi}(\Omega).
\end{align*}
Here, the interpolation parameter $\omega$ is defined by the formulas
\begin{equation*}
\omega(r):=\alpha(r)(\beta(r)/\alpha(r))^{1/2}=r^{1/2}\varphi(r^{1/(2\varepsilon+2\delta)})
\quad\mbox{if}\quad r\geq1
\end{equation*}
and $\omega(r):=1$ if $0<r<1$. Thus, \eqref{8f73} is valid.
\end{proof}

\begin{remark}\label{8rem2}
The spaces defined by formulas \eqref{8f71} and \eqref{8f72} are
independent of the choice of the number $\varepsilon\in(0,1/2)$ up to equivalence of norms. Indeed, let
$l\in\{0,\,1\}$, $s\in E_{l}$; then according to Theorems \ref{8th1} and \ref{8th2} we have the isomorphisms
\begin{equation*}
\Lambda_l:H^{s,s/2;\varphi}(\Omega)\leftrightarrow
\bigl[\mathcal{Q}_{l}^{s-2-\varepsilon,s/2-1-\varepsilon/2;\varphi},
\mathcal{Q}_{l}^{s-2+\varepsilon,s/2-1+\varepsilon/2;\varphi}
\bigr]_{1/2}.
\end{equation*}
whenever $0<\varepsilon<1/2$. This means the required independence.
\end{remark}

\end{document}